\documentclass[12pt,oneside]{amsart}
\usepackage[margin=1in]{geometry}
\usepackage{amsmath,amssymb,amsthm,enumitem}
\usepackage[hidelinks,hypertexnames=false]{hyperref}
\setlist{itemsep=.06125in}
\numberwithin{equation}{section}
\theoremstyle{plain}
\newtheorem{theorem}{Theorem}[section]
\newtheorem{lemma}[theorem]{Lemma}
\newtheorem{corollary}[theorem]{Corollary}
\newtheorem{proposition}[theorem]{Proposition}
\theoremstyle{definition}
\newtheorem{definition}[theorem]{Definition}
\theoremstyle{remark}
\newtheorem{remark}[theorem]{Remark}
\newtheorem{notation}[theorem]{Notation}
\newtheorem*{theoremA}{Theorem A}
\newtheorem*{theoremB}{Theorem B}
\newtheorem*{theoremC}{Theorem C}
\newcommand{\R}{\mathbb{R}}
\newcommand{\N}{\mathbb{N}}
\newcommand{\Z}{\mathbb{Z}}
\newcommand{\eps}{\varepsilon}

\date{}

\title{Point Configurations in Brownian Traces}

\author{P. Kosenko}
\address{Department of Mathematics, University of British Columbia, Vancouver, BC, Canada}
\email{pkosenko@math.ubc.ca}

\author{J.M. Medina}
\address{Universidad de Buenos Aires, IAM-CONICET, and Universidad Cat\'{o}lica Argentina}
\email{jmedina@conicet.gov.ar}

\author{A. Yavicoli}
\address{Department of Mathematics, University of British Columbia, Vancouver, BC, Canada}
\email{yavicoli@math.ubc.ca}


\begin{document}

\begin{abstract}
We study finite and countable point configurations in the trace of Brownian
motion in $\R^d$, $d\geq2$.  In the plane, the set of translations placing a
prescribed finite configuration on the trace has Hausdorff dimension two in
every nonempty open set; we prove analogous dimension-preservation results
for radial and bilipschitz patterns, and every prescribed countable
configuration has a dense translation set.  For $d\geq3$ we prove a sharp
occurrence criterion for regular scale-invariant families of configurations,
in terms of the ranks of all labelled subconfigurations.  It gives exact
thresholds for arithmetic progressions and for homothetic and similar copies
of finite sets.  In particular we recover and extend a result of Benjamini
and Kozma \cite{BK}: almost surely the longest arithmetic progression in the
trace has length five in $\R^3$, three in $\R^4$ and $\R^5$, and two in
$\R^d$ for $d\geq6$.
\end{abstract}

\maketitle

\tableofcontents

\section{Introduction}
\label{section:introduction}

Every set of positive Lebesgue measure in $\R^d$ contains a homothetic copy
$a+tF$, $t>0$, of each prescribed finite set $F$: restricting to a subset
$E$ of finite positive measure, continuity of translations in $L^1$ gives
$\int\prod_{f\in F}\mathbf 1_E(a+tf)\,da\to\operatorname{Leb}(E)>0$ as
$t\downarrow0$.  For Lebesgue-null sets, size alone does not decide whether
a pattern occurs.  Falconer constructed, for each nondegenerate triangle, a
compact planar set of full Hausdorff dimension containing no similar copy of
it \cite{FalconerAvoid}; Keleti constructed compact subsets of
$\R$ of full dimension avoiding similar copies of any countable family of
patterns with at least three points \cite[Cor.~1.2]{Keleti}; polynomial,
nonlinear and linear avoidance results are due to M\'ath\'e, Fraser and
Pramanik, and Yavicoli \cite[Thm.~2.3]{Mathe},
\cite[Thms.~1.1--1.3]{FraserPramanikAvoid}, \cite[Thm.~A]{YavicoliAvoid}.
Conversely, very thin perfect sets can contain every finite polynomial
pattern, or arbitrarily long arithmetic progressions
\cite[Thms.~3.8--3.9]{MolterYavicoli}, \cite[Thm.~1.1]{Lai}.  Fourier
dimension does not decide the question either: Shmerkin constructed a
full-dimensional Salem subset of $\R/\Z$ without three-term arithmetic
progressions \cite[Thm.~B]{ShmerkinAP}, and Denson extended this to
nonlinear configurations \cite[Thms.~1.2--1.4]{Denson}.

Positive results need additional structure.  {\L}aba and Pramanik proved
that a measure on $[0,1]$ with sufficiently large Hausdorff dimension, in
the sense of a Frostman condition, and sufficiently fast Fourier decay
supports a three-term progression \cite[Thm.~1.2]{LabaPramanik}; this was
extended to more general linear and polynomial configurations
\cite[Thm.~1.6]{ChanLabaPramanik}, \cite[Thm.~1.1]{FraserGuoPramanik}, and
under related hypotheses Carnovale and Senger showed that the set of
lengths of common differences of three-term progressions in the support
has positive Hausdorff measure in an explicit range of dimensions, and
positive Lebesgue measure in dimension $d\ge2$
\cite[Cor.~1.6]{CarnovaleSenger}.  Large Hausdorff dimension alone forces
certain curved configurations \cite[Thm.~1]{KucaOrponenSahlsten},
\cite[Thm.~1.2]{Zhu}, and large Newhouse thickness forces long
progressions and homothetic copies of finite sets
\cite[Thm.~4]{YavicoliThick}.

Random fractals are a natural setting for sharp answers.  Shmerkin and
Suomala found thresholds for finite patterns, angles, distances and simplex
volumes in fractal percolation \cite[Thm.~1.1]{ShmerkinSuomala}: for a
fractal-percolation set in $\R^d$ of almost-sure dimension $s$, conditional
on nonextinction, a homothetic copy of a fixed $k$-point set, $k\ge2$,
occurs when
\[
 k(d-s)<d+1
\]
\cite[Cor.~5.9]{ShmerkinSuomala}, and the reverse inequality gives
almost-sure avoidance of each fixed pattern
\cite[Cor.~5.14]{ShmerkinSuomala}.  With $s=2$ this is exactly the
homothetic threshold $k(d-2)<d+1$ obtained below for Brownian traces.  The
Brownian trace is nevertheless a very different random fractal: it is
generated by a continuous path, has a different dependence and intersection
structure, and has Hausdorff dimension two in every ambient dimension
$d\geq2$.  Beyond this numerical parallel, we show that it admits a sharp
criterion, in terms of subconfiguration ranks, for general scale-invariant
families of patterns.

Let $W=(W(t))_{t\geq0}$ be Brownian motion in $\R^d$ started at the origin
and let $B=W([0,\infty))$ be its trace.  The following facts hold almost
surely.  In every dimension $d\geq2$, $\dim_HB=2$ \cite[Thm.~4.29(b)]{PM}.
In the plane, $B$ is dense by neighbourhood recurrence
\cite[Thm.~3.20 and Rem.~3.21]{PM} and Lebesgue null \cite[Thm.~2.24]{PM},
hence not closed.  For $d\geq3$, transience gives $|W(t)|\to\infty$
\cite[Thm.~3.20]{PM}, so $B$ is unbounded and, by path continuity, closed.
Planar recurrence and spatial transience thus lead to different questions
about patterns in the trace.

Benjamini and Kozma proved that the planar trace contains arithmetic
progressions of every finite length, and that the trace in $\R^3$ contains
five-term progressions but no six-term progression
\cite[Prop.~1 and Lemmas~2 and 4]{BK}.  In the plane we describe the
parameter sets of these progressions and treat general finite and countable
configurations.  In space we determine the maximal length in every dimension
and place it within a criterion for general scale-invariant families.

\medskip\noindent\emph{The plane.}\ 
Our main planar input is the dimension formula
\[
 \dim_H(E\cap B)=\dim_H E\qquad\text{almost surely}
\]
for every fixed Borel set $E\subseteq\R^2$ (Lemma~\ref{Lemma:dimBrownie}).
It follows from Hawkes' preimage theorem \cite{Hawkes}, in the form
\cite[Eq.~(1.9)]{KX}, which gives the dimension of the time set
$\{t:W(t)\in E\}$, and from Kaufman's uniform dimension-doubling theorem
\cite{Kaufman}, \cite[Thm.~9.28]{PM}, which applies simultaneously to all
time sets and hence to this random one.  To place several prescribed points
on the trace we reveal the path up to a stopping time and impose the next
condition on its continuation, which by the strong Markov property
\cite[Thm.~2.16]{PM} is a new Brownian motion independent of the past; each
condition can be imposed without decreasing the Hausdorff dimension of the
parameter set (Lemma~\ref{lem:intersection}).  We
apply this to patterns $x,\lambda_1x,\ldots,\lambda_kx$ with $x$ in a
radial set (Theorem~\ref{thm:radial}), to sets linked by bilipschitz maps
(Theorem~\ref{theo:bilip}) and to translates: for every finite
$F\subset\R^2$, almost surely
\[
 \dim_H(T_F\cap U)=2,
 \qquad
 T_F:=\{a\in\R^2:a+F\subseteq B\},
\]
for every nonempty open $U\subseteq\R^2$
(Corollary~\ref{cor:planartranslates}).

For arithmetic progressions $a,a+\vec d,\ldots,a+(n-1)\vec d$ we also vary
the common difference, or gap vector, $\vec d$.  For every $n\geq2$, the set
of pairs $(a,\vec d)$ giving an $n$-term progression on $B$ has Hausdorff
dimension four in every nonempty open subset of $\R^4$; moreover, on one
probability-one event, for Lebesgue-almost every $\vec d$ and every
$n\geq2$ the set $\{a\in U:a+j\vec d\in B\text{ for }0\leq j<n\}$ has
Hausdorff dimension two for every nonempty open $U\subseteq\R^2$
(Theorem~\ref{thm:planarAdim}).  The starting point may depend on $n$, so
this does not produce an infinite progression.  Infinitely many conditions
are handled by the same construction: for every fixed countable
$F\subseteq\R^2$, $T_F$ is almost surely dense
(Theorem~\ref{prop:countable}).  Taking $F=\{j\vec d:j\geq0\}$, for each
fixed $\vec d\neq0$ there is almost surely a dense set of starting points of
infinite progressions with common difference $\vec d$, and on one
probability-one event this holds for Lebesgue-almost every $\vec d$.  The
parameter set $\mathcal A_\infty$ of nondegenerate infinite progressions is
dense in $\R^4$ and satisfies $2\leq\dim_H\mathcal A_\infty\leq4$
(Theorem~\ref{thm:infiniteAP}).

\medskip\noindent\emph{Dimensions $d\geq3$.}\ 
The one-ball hitting formula \cite[Cor.~3.19]{PM} says that a Brownian path
starting at distance comparable to one from a ball of radius $\eps$ hits it
with probability comparable to $\eps^{d-2}$.  Requiring the path to enter
$l$ well-separated prescribed balls therefore costs a factor comparable to
$\eps^{l(d-2)}$; this is the source of the exponent $d-2$ below.  (A
different expression of the two-dimensionality of the spatial trace is the
capacity equivalence of Pemantle, Peres and Shapiro: $W([0,1])$ has positive
$f$-capacity for exactly the same radial kernels $f$ as the unit square
\cite[Thm.~1.1]{PPS}.)

From dimension three on, occurrence of a family of patterns is governed by
the number of parameters of the whole pattern and of each of its labelled
subpatterns.  Locally label the points and parametrise the family by a
$C^1$ map
\[
 \Phi=(\phi_1,\ldots,\phi_k):
 U\subseteq\R^D\longrightarrow\operatorname{Conf}_k(\R^d),
\]
where $\operatorname{Conf}_k(\R^d)$ is the space of ordered $k$-tuples of
distinct points.  When $\Phi$ is a local embedding, $D$ is the number of
independent real parameters needed to place the pattern; it is not the
affine or Hausdorff dimension of the underlying finite set.  For $k$-term
arithmetic progressions,
$\Phi(a,\vec d)=(a,a+\vec d,\ldots,a+(k-1)\vec d)$ with
$(a,\vec d)\in\R^d\times(\R^d\setminus\{0\})$, so $D=2d$.  For the positive
homothetic copies $\Phi(a,t)=(a+tf_i)_{i=1}^k$, $a\in\R^d$, $t>0$, of a
fixed set $F=\{f_1,\ldots,f_k\}$ of $k\ge2$ distinct points, $D=d+1$.

For a $C^1$ map $G$ we write $DG(x)$ for its differential at $x$ and, when
the rank is constant, $\operatorname{rank}DG$ for this rank.  For a nonempty
set of labels $I$, let $\Phi_I=(\phi_i)_{i\in I}$ record the positions of
the points labelled by $I$, and put
\[
 r_I:=\operatorname{rank}D\Phi_I,
 \qquad
 \Delta_I:=r_I-(d-2)\,\# I .
\]
When the rank is constant, $r_I$ is the local dimension of the possible
placements of the selected points, and the \emph{excess} $\Delta_I$ compares
it with the probabilistic cost $(d-2)\,\# I$ of placing them on the trace.
For the full label set $I_0=\{1,\ldots,k\}$ we have $r_{I_0}=D$, and we call
$\Delta:=\Delta_{I_0}=D-k(d-2)$ the \emph{full excess}.  A positive full
excess is not sufficient: a proper subpattern may vary in too few dimensions
and have $\Delta_I\leq0$, and since every occurrence of the full pattern
realises that subpattern, it prevents the full pattern from occurring.

A family is \emph{scale invariant} if it is closed under
$(x_1,\ldots,x_k)\mapsto(\lambda x_1,\ldots,\lambda x_k)$ for every
$\lambda>0$.  We impose the following regularity, which excludes singular or
redundant parametrisations: the family is covered by countably many locally
labelled $C^1$ parametrisations $\Phi$ as above, each an embedding, each
$\phi_i$ of rank $d$, each $\Phi_I$ of constant rank on each local chart,
and, after relabelling if necessary, all charts have the same $D$ and the
same ranks $r_I$.  Dilation is one of the parameters, not a separate
coordinate.  Definition~\ref{def:atlas} calls this a \emph{rank-regular
fixed-rank scale atlas}; the progression, homothety and similarity families
satisfy it by direct rank computations.

\begin{theoremA}[sharp rank criterion; Theorem~\ref{thm:trichotomy} and
Corollary~\ref{cor:subcriterion}]
Let $d\geq3$ and let $\mathcal F$ be a fixed scale-invariant family of
$k$-point configurations in $\R^d$ with a rank-regular fixed-rank scale
atlas.  Almost surely the trace of Brownian motion contains a member of
$\mathcal F$ if
\[
 r_I>(d-2)\,\# I
 \qquad\text{for every nonempty }I\subseteq\{1,\ldots,k\}.
\]
If this inequality fails for some nonempty $I$, including at equality, then
almost surely the trace contains no member of $\mathcal F$.
\end{theoremA}

In the applications below, positivity of the full excess implies positivity
for all proper subpatterns, so a single inequality has to be checked.  We
also prove a one-sided dimension bound: on each normalised rank chart
(Definition~\ref{def:family}) with $\Delta\geq0$, the set of parameters
producing patterns on the trace has
Hausdorff and upper Minkowski dimension at most $\Delta$
(Proposition~\ref{prop:realisedupper}).

For progressions, a single term has rank $d$ and any two distinct terms
determine $(a,\vec d)$, so every subpattern with at least two points has
rank $2d$.  All inequalities in Theorem~A therefore hold exactly when
$2d-k(d-2)>0$.

\begin{theoremB}[longest progression; Corollary~\ref{cor:APall}]
Almost surely, for $d\geq3$, the longest arithmetic progression in the trace
of Brownian motion in $\R^d$ has length $\lceil2d/(d-2)\rceil-1$: five in
$\R^3$, three in $\R^4$ and $\R^5$, and two in $\R^d$ for $d\ge6$.  Almost
surely the planar trace contains progressions of every finite length.
\end{theoremB}

Similarly, any two labels of a positive homothetic pattern determine $a$ and
$t$, so every subpattern with at least two points has rank $d+1$.  Hence,
for every fixed set of $k\geq2$ distinct points, the trace almost surely
contains a positive homothetic copy exactly when $k<(d+1)/(d-2)$
(Corollary~\ref{cor:homothetic}).

Finally we allow rotations.  For a finite set
$F=\{f_1,\ldots,f_k\}\subset\R^d$ let
\[
 m=\dim\operatorname{span}\{f_i-f_1:2\leq i\leq k\}
\]
be its affine dimension.  Translation and dilation contribute $d+1$
parameters and the rotation group $\binom d2$; after translating $f_1$ to
the origin, the rotations fixing the $m$-dimensional direction space of $F$
pointwise do not move the pattern and account for $\binom{d-m}{2}$
redundant parameters.  Hence the number of independent parameters of a
similar copy is
\[
 D_m=d+1+\binom d2-\binom{d-m}{2}
     =d+1+\tfrac12m(2d-m-1).
\]
Allowing reflections adds finitely many choices and does not change this
dimension.

\begin{theoremC}[similar copies; Theorem~\ref{thm:similar}]
Let $d\ge3$ and let $F\subset\R^d$ be a set of $k\geq2$ distinct points of
affine dimension $m\geq1$.  Almost surely the trace of Brownian motion
contains a similar copy of $F$ if $k<D_m/(d-2)$, and almost surely contains
none if $k\geq D_m/(d-2)$.
\end{theoremC}

No subset of $F$ gives a stronger restriction: a subset $I$ with $n\geq2$
points and affine dimension $e$ has rank $r_I=d+1+\frac12e(2d-e-1)$ and
satisfies $n\leq k-(m-e)$, from which Section~\ref{subsec:similar} deduces
$\Delta_I\geq\Delta+\frac12(m-e)(m+e-3)\geq\Delta$; singletons have excess
$2$.  We also note that the distance set $\{|x-y|:x,y\in B\}$ of the
spatial trace is almost surely $[0,\infty)$ (Remark~\ref{rem:distance}).

\medskip\noindent\emph{Method and organisation.}\ 
In the plane, the intersection formula and the strong Markov property let us
place the prescribed points one at a time on successive portions of the
path.  In dimensions $d\geq3$ we count approximate patterns at resolution
$\eps$: ball-hitting estimates give the expected number of candidates, the
ranks $r_I$ bound the number of pairs of candidates sharing nearby labelled
points, and a second-moment argument gives occurrence with positive
probability; compactness turns approximate patterns into an exact one, and
ergodicity of Brownian scaling upgrades positive probability to probability
one.  For nonoccurrence, a product capacity defined by an explicit energy
shows that when $\Delta_I\leq0$ the allowed positions of the points
labelled by $I$ are not reached by independent Brownian paths, and Brownian
bridges transfer this to a single trace.

Section~\ref{section:plane} treats the plane.  Section~\ref{section:Rd}
develops the spatial rank criterion: the second-moment argument in positive
excess, the local parameter-dimension bound, and the capacity argument in
nonpositive excess.  Section~\ref{section:consequences} applies the
criterion to progressions, homothetic and similar copies, corners and the
distance set.

\section{Patterns in the trace of planar Brownian motion}
\label{section:plane}

\subsection{Preliminaries}

For $x\in\R^d$ let $\Omega=C([0,\infty),\R^d)$, $W(t,\omega)=\omega(t)$,
and let $\mathbb P_x$ be Wiener measure on $\Omega$ concentrated on paths
starting at $x$; under $\mathbb P_x$ the coordinate process
$W=(W(t))_{t\geq0}$ has continuous paths and independent increments with
$W(t)-W(s)\sim N(0,(t-s)I_d)$ for $0\le s<t$.  We write $\mathbb E_x$ for
expectation and omit the subscript when $x=0$.  Let $(\mathcal F_t)_{t\geq0}$
be the usual augmentation of the natural filtration
$\sigma(W(s):0\le s\le t)$; it is right-continuous and contains every subset
of a null event.  Put $\mathcal F=\sigma(\bigcup_{t\ge0}\mathcal F_t)$; a
statement holds \emph{almost surely} if it fails only on a $\mathbb P_x$-null
event.  A random time $\tau$ is a \emph{stopping time} if
$\{\tau\le t\}\in\mathcal F_t$ for every $t$, and
\[
 \mathcal F_\tau
 :=\{A\in\mathcal F:A\cap\{\tau\le t\}\in\mathcal F_t
                         \text{ for every }t\ge0\}
\]
is the information available at time $\tau$.  The strong Markov property
states that for every almost surely finite stopping time $\tau$ the process
$(W(\tau+t)-W(\tau))_{t\ge0}$ is a Brownian motion independent of
$\mathcal F_\tau$ \cite[Thm.~2.16]{PM}.  Brownian scaling states that for
every $\lambda>0$, under $\mathbb P_x$ the process
$\bigl(x+\lambda^{-1}(W(\lambda^2t)-x)\bigr)_{t\ge0}$ has the law of
$(W(t))_{t\ge0}$ \cite[Lem.~1.7]{PM}.

In this section $d=2$, the path starts at the origin, and
\[
 B:=W([0,\infty)).
\]
For $E\subseteq\R^n$, $s\ge0$ and $\delta>0$, let
\[
 \mathcal H^s_\delta(E):=\inf\Bigl\{\sum_i(\operatorname{diam}U_i)^s:
 E\subseteq\bigcup_iU_i,\ \operatorname{diam}U_i\le\delta\Bigr\},
\]
the infimum being over countable covers.  The $s$-dimensional
\emph{Hausdorff measure} of $E$ is
$\mathcal H^s(E):=\lim_{\delta\downarrow0}\mathcal H^s_\delta(E)$, which
exists because $\mathcal H^s_\delta(E)$ increases as $\delta$ decreases;
here $0^0:=1$, so that $\mathcal H^0$ is counting measure.  The
\emph{Hausdorff dimension} of $E$ is
\[
 \dim_HE:=\inf\{s\ge0:\mathcal H^s(E)=0\}
 =\sup\{s\ge0:\mathcal H^s(E)=\infty\},
\]
with $\sup\varnothing=0$; in particular $\dim_H\varnothing=0$.  Hausdorff
dimension is monotone, invariant under bilipschitz maps, and countably
stable:
\begin{equation}\label{eq:countablestability}
 \dim_H\Bigl(\bigcup_{n=1}^{\infty}E_n\Bigr)
 =\sup_{n\geq1}\dim_H E_n ;
\end{equation}
see \cite[Chapter~3]{Falconer}.  Here a map $g$ between subsets of
Euclidean spaces is \emph{bilipschitz} if
$c|x-y|\le|g(x)-g(y)|\le C|x-y|$ for some constants $0<c\le C<\infty$.
A set is \emph{deterministic} if it is fixed before the Brownian path is
sampled.

\begin{lemma}[planar intersection formula]\label{Lemma:dimBrownie}
For every deterministic Borel set $E\subseteq\R^2$,
\[
\dim_H(E\cap B)=\dim_H E
\qquad\text{almost surely}.
\]
\end{lemma}

\begin{proof}
We may assume $E\ne\varnothing$.  Put
$T_E:=W^{-1}(E)=\{t\geq0:W(t)\in E\}$.  Hawkes' preimage theorem
\cite{Hawkes}, in the form \cite[Eq.~(1.9)]{KX}, applied to planar
Brownian motion, the isotropic $2$-stable process in $\R^2$, gives
\[
 \dim_HT_E=\tfrac12\dim_HE
 \qquad\text{almost surely}.
\]
Kaufman's uniform
dimension-doubling theorem \cite{Kaufman}, \cite[Thm.~9.28]{PM} gives an
event $\Omega_K$ of probability one on which $\dim_HW(A)=2\dim_HA$ for every
$A\subseteq[0,1]$ simultaneously.  Applying it to the shifted path on each
interval $[m,m+1]$, intersecting the countably many events and using
\eqref{eq:countablestability}, we obtain $\dim_HW(A)=2\dim_HA$ for every
$A\subseteq[0,\infty)$ on one probability-one event.  Since
$W(T_E)=E\cap B$, taking $A=T_E$ proves the formula.  The simultaneous
quantifier is essential, because $T_E$ depends on the sampled path.
\end{proof}

The exceptional null set in Lemma~\ref{Lemma:dimBrownie} may depend on
$E$.  The next lemma, based on the strong Markov property, is what makes
the inductive constructions below possible.  A \emph{random compact set}
$K$, possibly empty, is measurable if $\{K\cap U\ne\varnothing\}$ is
measurable for every open $U\subseteq\R^2$; equivalently, it is a
measurable map into the hyperspace of compact sets with the
Hausdorff-metric Borel structure, the empty set being an isolated point.
Finite-time Brownian ranges, their continuous images and compact
intersections are random compact sets.

\begin{lemma}[intersection lemma]\label{lem:intersection}
Let $\tau$ be an almost surely finite stopping time and let $K$ be an
$\mathcal F_\tau$-measurable random compact subset of $\R^2$.  If
$q\geq0$ and $\dim_H K>q$ almost surely, then there is a stopping time
$\sigma$ with $\mathbb P(\tau<\sigma<\infty)=1$ such that
\[
K':=K\cap W([\tau,\sigma])
\]
is an $\mathcal F_\sigma$-measurable random compact set with
$\dim_H K'>q$ almost surely.
\end{lemma}

\begin{proof}
Since the filtration is complete, we may modify $\tau$ and $K$ on the null
event $\{\tau=\infty\}$ and assume $\tau<\infty$ everywhere.  Put
$\widetilde W(t)=W(\tau+t)-W(\tau)$, a planar Brownian motion independent
of $\mathcal F_\tau$ by the strong Markov property.

\emph{Conditioning.}  For a compact set $L$ and a continuous path $\gamma$
put
\[
 D(L,\gamma):=\sup_{m\ge1}\dim_H\bigl(L\cap\gamma([0,m])\bigr)
 =\dim_H\bigl(L\cap\gamma([0,\infty))\bigr),
\]
the equality by \eqref{eq:countablestability}.  This is Borel in
$(L,\gamma)$: finite path ranges are compact, intersection is a Borel
operation on the hyperspace, and Hausdorff dimension is a Borel function
there \cite[Thm.~2.1(b)]{MattilaMauldin}.  Hence
$h(L):=\mathbb P\{D(L,\widetilde W)\ne\dim_HL\}$ is a Borel function of
$L$, and Lemma~\ref{Lemma:dimBrownie} says that $h(L)=0$ for every compact
$L$.  Since $K-W(\tau)$ is $\mathcal F_\tau$-measurable and
$\widetilde W$ is independent of $\mathcal F_\tau$,
\[
 \mathbb P\bigl(
  D(K-W(\tau),\widetilde W)\ne\dim_H(K-W(\tau))
 \bigm|\mathcal F_\tau\bigr)
 =h\bigl(K-W(\tau)\bigr)=0
 \quad\text{almost surely}.
\]
As $(K\cap W([\tau,\infty)))-W(\tau)=(K-W(\tau))\cap\widetilde W([0,\infty))$
and Hausdorff dimension is translation invariant, this says that
$\dim_H(K\cap W([\tau,\infty)))=\dim_HK>q$ almost surely.

\emph{Stopping.}  For $m\ge1$ put $R_m:=K\cap W([\tau,\tau+m])$.  These
are increasing $\mathcal F_{\tau+m}$-measurable random compact sets (the
range over $[\tau,\tau+m]$ is determined by $\tau$ and the path stopped at
$\tau+m$) with union $K\cap W([\tau,\infty))$, so by
\eqref{eq:countablestability},
$\sup_m\dim_HR_m=\dim_H(K\cap W([\tau,\infty)))>q$ almost surely, and
hence $\dim_HR_m>q$ for some finite $m$; this is where the strict
inequality $\dim_HK>q$ is used.  Let
$N_0:=\min\{m\ge1:\dim_HR_m>q\}$, with $\min\varnothing=\infty$.  The
event $A:=\{N_0=\infty\}$ is null, hence lies in $\mathcal F_0$ by
completeness.  Put $N:=N_0$ on $A^c$, $N:=1$ on $A$, $\sigma:=\tau+N$ and
$K':=R_N$.  Then $\sigma<\infty$ everywhere and $\dim_HK'>q$ outside $A$.
For $m\ge1$, $\{N_0\le m\}=\{\dim_HR_m>q\}\in\mathcal F_{\tau+m}$ and
$A\in\mathcal F_0\subseteq\mathcal F_{\tau+m}$, so
$\{N\le m\}=(\{N_0\le m\}\cap A^c)\cup A\in\mathcal F_{\tau+m}$ and
$\{N=m\}=\{N\le m\}\setminus\{N\le m-1\}\in\mathcal F_{\tau+m}$.  Hence
\[
 \{\sigma\le t\}
 =\bigcup_{m\ge1}\bigl(\{N=m\}\cap\{\tau+m\le t\}\bigr)
 \in\mathcal F_t ,
\]
so $\sigma$ is a stopping time, and for every open $U\subseteq\R^2$,
\[
 \{K'\cap U\ne\varnothing,\ \sigma\le t\}
 =\bigcup_{m\ge1}
 \{N=m,\ \tau+m\le t,\ R_m\cap U\ne\varnothing\}
 \in\mathcal F_t ,
\]
so $K'$ is $\mathcal F_\sigma$-measurable.  Finally $\sigma>\tau$ because
$N\ge1$.
\end{proof}

Countable intersections of probability-one events, for instance over all
pattern lengths or over rational $q<\dim_HE$, are used below without
further comment.

\subsection{Radial patterns}

For $\theta\in S^1$ let $L_\theta=\{r\theta:r>0\}$.  A Borel set
$C\subseteq\R^2\setminus\{0\}$ is \emph{radial} if $C\cap L_\theta$ is a
single point for every $\theta\in S^1$.

\begin{theorem}[radial patterns]\label{thm:radial}
Let $C$ be a deterministic Borel radial set, let
$\lambda_1,\ldots,\lambda_k>0$, and put $\lambda_0=1$.  Then almost surely
\[
X(C;\lambda_1,\ldots,\lambda_k)
:=\bigl\{x\in C:\lambda_jx\in B\text{ for }0\leq j\leq k\bigr\}
\]
has Hausdorff dimension $\dim_H C$.  If moreover the radial projection
$\pi|_C:C\to S^1$, $\pi(x)=x/|x|$, is bilipschitz, then the set of
directions $\pi(X(C;\lambda_1,\ldots,\lambda_k))$ also has Hausdorff
dimension $\dim_H C$.
\end{theorem}

\begin{proof}
If $\dim_HC=0$ the claim follows from
$X(C;\lambda_1,\ldots,\lambda_k)\subseteq C$.  Let $\dim_HC>0$, fix
$0\le q<\dim_HC$ and choose $q<s<\dim_HC$.  Frostman's lemma
\cite{Falconer} gives a nonzero finite Borel measure $\nu$
carried by $C$ with $\nu(B(x,r))\le Cr^s$, and by inner regularity some
compact $K_0\subseteq C$ has $\nu(K_0)>0$; the ball estimate for
$\nu|_{K_0}$ and the mass distribution principle \cite[Chapter~4]{Falconer}
give $\dim_HK_0\ge s>q$.  Put
$\tau_0=0$, $\mu_1=1$ and $\mu_{j+1}=\lambda_j$ for $1\le j\le k$.  Given
a compact $\mathcal F_{\tau_{j-1}}$-measurable $K_{j-1}\subseteq C$ of
dimension greater than $q$, where $1\le j\le k+1$, apply
Lemma~\ref{lem:intersection} after time $\tau_{j-1}$ to $\mu_jK_{j-1}$,
which is an $\mathcal F_{\tau_{j-1}}$-measurable random compact set of the
same dimension because multiplication by $\mu_j$ is bilipschitz.  This
gives a finite stopping time $\tau_j>\tau_{j-1}$ for which
\[
K_j:=\mu_j^{-1}
   \bigl(\mu_jK_{j-1}\cap W([\tau_{j-1},\tau_j])\bigr)
\]
is a compact subset of $K_{j-1}$ of dimension greater than $q$.  If
$x\in K_{k+1}$ then $\mu_jx\in B$ for $1\le j\le k+1$, that is,
$\lambda_jx\in B$ for $0\le j\le k$.  Hence
$K_{k+1}\subseteq X(C;\lambda_1,\ldots,\lambda_k)$ and
$\dim_HX(C;\lambda_1,\ldots,\lambda_k)>q$.  Intersecting over rational
$q<\dim_HC$ and using $X(C;\lambda_1,\ldots,\lambda_k)\subseteq C$ gives
equality.  The last assertion follows from bilipschitz invariance of
Hausdorff dimension, since $\pi|_X$ is a restriction of $\pi|_C$.
\end{proof}

The bilipschitz hypothesis cannot be dropped, since radial projection need
not preserve Hausdorff dimension for a general radial set; it holds for a
circle centred at the origin.

\begin{corollary}\label{cor:chain}
For any $d_1,\ldots,d_k>0$, almost surely the set of directions
$\theta\in S^1$ with $d_i\theta\in B$ for $1\leq i\leq k$ has Hausdorff
dimension one.
\end{corollary}

\begin{proof}
Apply Theorem~\ref{thm:radial} with $C=S^1$ and $\lambda_i=d_i$; on $S^1$
the radial projection is the identity.
\end{proof}

\begin{corollary}\label{cor:planarAP}
Fix $\lambda>0$.  Almost surely, simultaneously for every $m\geq1$, the set
of directions $\theta\in S^1$ with
\[
 0,\lambda\theta,2\lambda\theta,\ldots,m\lambda\theta\in B
\]
has Hausdorff dimension one.  In particular $B$ almost surely contains
arbitrarily long arithmetic progressions starting at the origin whose
common difference has length $\lambda$.
\end{corollary}

\begin{proof}
For each $m$ apply Corollary~\ref{cor:chain} with $d_i=i\lambda$,
$1\le i\le m$; since $W(0)=0$, the points
$0,\lambda\theta,\ldots,m\lambda\theta$ form a progression on $B$.
Intersect over $m$.  The direction may depend on $m$.
\end{proof}

\subsection{Patterns in bilipschitz-related sets}

Let $E_1,\ldots,E_k$ be deterministic Borel sets and, for $1\le i<k$, let
$g_i:E_i\to E_{i+1}$ be a fixed bilipschitz bijection.  Write
\[
g_{i,j}:=g_{j-1}\circ\cdots\circ g_i\quad(i<j),
\qquad g_{j,j}:=\operatorname{id}_{E_j}.
\]

\begin{theorem}\label{theo:bilip}
Almost surely the set
\[
S:=\bigl\{y\in E_k:g_{i,k}^{-1}(y)\in B
                      \text{ for every }1\leq i\leq k\bigr\}
\]
has Hausdorff dimension $\dim_H E_1$.
\end{theorem}

\begin{proof}
All $E_i$ have the same dimension, so if $\dim_HE_1=0$ the claim follows
from $S\subseteq E_k$.  Let $\dim_HE_1>0$, fix $0\le q<\dim_HE_1$ and, as
in the previous proof, choose a compact $D_1\subseteq E_1$ with
$\dim_HD_1>q$.  Lemma~\ref{lem:intersection} applied to $D_1$ at time
zero gives a compact $K_1\subseteq D_1\cap B$ of dimension greater than
$q$, measurable at a finite stopping time $\tau_1$.  Inductively, given
$K_{j-1}\subseteq E_{j-1}\cap B$, apply the lemma after $\tau_{j-1}$ to
$g_{j-1}(K_{j-1})$, which is an $\mathcal F_{\tau_{j-1}}$-measurable random
compact set of dimension greater than $q$ because $g_{j-1}$ is fixed,
continuous and bilipschitz.  This gives a compact
$K_j\subseteq g_{j-1}(K_{j-1})\cap B$ with $\dim_HK_j>q$, measurable at a
finite stopping time $\tau_j$.  For $y\in K_k$ put $y_i=g_{i,k}^{-1}(y)$;
the inclusions $K_j\subseteq g_{j-1}(K_{j-1})$ give, by backward induction,
$y_i\in K_i\subseteq B$ for every $i$.  Hence $K_k\subseteq S$ and
$\dim_HS>q$.  Intersecting over rational $q$ gives
$\dim_HS\ge\dim_HE_1$; the reverse inequality follows from
$S\subseteq E_k$.
\end{proof}

A \emph{rational closed disk} is a closed disk with rational centre and
positive rational radius; there are countably many, and every nonempty open
subset of $\R^2$ contains one.

\begin{corollary}\label{cor:planartranslates}
For every deterministic finite set $F\subseteq\R^2$, almost surely the
translation set
\[
T_F:=\{a\in\R^2:a+F\subseteq B\}
\]
is locally of full dimension: $\dim_H(T_F\cap V)=2$ for every nonempty open
set $V\subseteq\R^2$.
\end{corollary}

\begin{proof}
Write $F=\{x_1,\ldots,x_k\}$; the empty case is trivial.  Fix a rational
closed disk $Q$, put $E_i=x_i+Q$ and $g_i(z)=z+(x_{i+1}-x_i)$.
Theorem~\ref{theo:bilip} gives $S\subseteq E_k$ of dimension two such that
$g_{i,k}^{-1}(y)=y+(x_i-x_k)\in B$ for all $y\in S$ and $1\le i\le k$;
thus $S-x_k\subseteq T_F\cap Q$ and $\dim_H(T_F\cap Q)=2$.  Intersect over
all rational closed disks.
\end{proof}

\subsection{Starting points, gap vectors and countable configurations}
\label{subsec:planarquestions}

We use the following slicing bound; compare \cite{Falconer}.

\begin{lemma}[fibres and dimension]\label{lem:fibre}
Let $s,t\geq0$, let $A\subseteq\R^m\times\R^n$ be Borel, and for
$y\in\R^n$ write $A_y:=\{x\in\R^m:(x,y)\in A\}$.  Suppose there is a Borel
set $Y\subseteq\R^n$ with $\mathcal H^s(Y)>0$ such that $A_y\ne\varnothing$
and $\dim_HA_y\geq t$ for every $y\in Y$.  Then $\dim_HA\geq s+t$.
\end{lemma}

\begin{proof}
If $s=0$, pick $y\in Y$: $A$ contains the isometric copy
$A_y\times\{y\}$, so $\dim_HA\geq t$.  If $t=0$, the Lipschitz projection
to the second coordinate maps $A$ onto a set containing $Y$, so
$\dim_HA\geq s$.  Let $s,t>0$ and fix $0<s'<s$, $0<t'<t$.  Frostman's
lemma \cite{Falconer} gives a nonzero finite Borel measure $\nu$
carried by $Y$ with $\nu(B(y,r))\le Cr^{s'}$ for all $y$ and $r$.  Suppose
$\mathcal H^{s'+t'}(A)=0$.  Given $\delta,\eta>0$, cover $A$ by balls
$B_i\subset\R^{m+n}$ of radii $r_i<\delta$ with
$\sum_ir_i^{s'+t'}<\eta$.  For each $y$ the nonempty sections $(B_i)_y$
cover $A_y$ and have diameter at most $2r_i$, so
\[
 \int\mathcal H^{t'}_{2\delta}(A_y)\,d\nu(y)
 \le2^{t'}\sum_ir_i^{t'}\nu(\pi_2(B_i))
 \le C'\sum_ir_i^{s'+t'}<C'\eta ,
\]
where the left-hand side is an upper integral; the pointwise Borel majorant
$2^{t'}\sum_ir_i^{t'}\mathbf 1_{\pi_2(B_i)}(y)$ is all that is needed.
Letting $\eta\downarrow0$ and then $\delta\downarrow0$ gives
$\mathcal H^{t'}(A_y)=0$ for $\nu$-almost every $y$, contradicting
$\dim_HA_y\ge t>t'$ on $Y$.  Hence $\dim_HA\geq s'+t'$; let
$s'\uparrow s$ and $t'\uparrow t$.
\end{proof}

Fix $n\geq2$ and define
\[
\mathcal A_n
:=\bigl\{(a,\vec d)\in\R^2\times(\R^2\setminus\{0\}):
          a+j\vec d\in B\text{ for }0\leq j\leq n-1\bigr\},
\]
the set of parameters of $n$-term progressions on $B$ with starting point
$a$ and common difference, or \emph{gap vector}, $\vec d$.  We call the
first factor $\R^2$, in which $a$ varies, the \emph{translation space}, and
the second, in which $\vec d$ varies, the \emph{gap space}.  For fixed
$\vec d\ne0$ the fibre is
\[
 (\mathcal A_n)_{\vec d}
 :=\{a\in\R^2:(a,\vec d)\in\mathcal A_n\}
 =\{a\in\R^2:a+j\vec d\in B\text{ for }0\leq j<n\}.
\]

\emph{Measurability.}  Let $\Omega_{\mathrm{path}}:=C_0([0,\infty),\R^2)$
be the space of continuous paths starting at the origin, with the topology
of uniform convergence on bounded time intervals, and regard Brownian motion
as a random element $\omega$ of it with law Wiener measure.  For every
$M\geq1$ the map $(\omega,x)\mapsto\min_{0\le t\le M}|\omega(t)-x|$ is
continuous, so
\[
 \mathfrak B:=\{(\omega,x):x\in B(\omega)\}
 =\bigcup_{M\ge1}
   \Bigl\{(\omega,x):\min_{0\le t\le M}|\omega(t)-x|=0\Bigr\}
\]
is Borel, and hence so is
\[
 \mathfrak A_n:=\{(\omega,a,\vec d):\vec d\neq0,
       \ a+j\vec d\in B(\omega),\ 0\le j\le n-1\}.
\]
We need two standard facts.  A subset of a Polish space is \emph{analytic}
if it is the projection of a Borel subset of a product with another Polish
space; every analytic set is universally measurable, that is, measurable
with respect to the completion of every finite Borel measure
\cite[Thm.~21.10]{Kechris}.  In particular such a set is measurable with
respect to the completed product of Wiener measure and Lebesgue measure on
a ball, so Tonelli's theorem for complete product measures applies to it.
Second, for fixed $(\omega,\vec d)$ the section
$(\mathcal A_n(\omega))_{\vec d}$ is a finite intersection of translates
of $B(\omega)$, hence a countable union of compact sets, and the
section-dimension theorem \cite[Thm.~6.1]{MattilaMauldin}, applied to the
Borel set $\mathfrak A_n$, shows that
$(\omega,\vec d)\mapsto\dim_H\bigl((\mathcal A_n(\omega))_{\vec d}\bigr)$
is Borel; the same holds after intersecting the section with a fixed open
set.  This justifies Tonelli's theorem in the next proof.  The
infinite-progression argument of Section~\ref{subsec:infinite}, whose
sections need not be countable unions of compact sets, uses the analytic-set
fact instead.

\begin{theorem}[finite progression parameters]\label{thm:planarAdim}
Almost surely, simultaneously for every $n\geq2$:
\begin{enumerate}
\item $\mathcal A_n$ is locally of full dimension:
$\dim_H(\mathcal A_n\cap O)=4$ for every nonempty open set
$O\subseteq\R^4$; in particular, $\mathcal A_n$ is dense in $\R^4$;
\item for Lebesgue-almost every $\vec d\in\R^2$, the fibre
$(\mathcal A_n)_{\vec d}$ is locally of full dimension:
$\dim_H\bigl((\mathcal A_n)_{\vec d}\cap U\bigr)=2$ for every nonempty
open set $U\subseteq\R^2$.
\end{enumerate}
\end{theorem}

Equivalently, for almost every sampled trace there is a Lebesgue-null set
$N_B\subseteq\R^2$, which we may take to contain $0$, such that (2) holds
for every $\vec d\notin N_B$ and every $n\geq2$.  One gap vector $\vec d\notin N_B$ thus supports
progressions of every finite length, but with starting points $a_n$ that
may depend on $n$; this does not give an infinite progression, which is the
subject of Section~\ref{subsec:infinite}.

\begin{proof}
Let $\mathcal U$ and $\mathcal V$ be the countable bases of open disks with
rational centres and positive rational radii in translation and gap space.
Fix $n\geq2$ and $U\in\mathcal U$.  For each fixed $\vec d\ne0$,
Corollary~\ref{cor:planartranslates} applied to the deterministic set
$\{0,\vec d,\ldots,(n-1)\vec d\}$ gives
$\mathbb P\bigl(\dim_H((\mathcal A_n)_{\vec d}\cap U)=2\bigr)=1$.  By the
measurability discussion above, Tonelli's theorem gives, for every $R>0$,
\[
 \mathbb E\Bigl[
 \operatorname{Leb}\bigl\{\vec d\in B(0,R):
 \dim_H((\mathcal A_n)_{\vec d}\cap U)<2\bigr\}
 \Bigr]
 =\int_{B(0,R)}
 \mathbb P\bigl(
 \dim_H((\mathcal A_n)_{\vec d}\cap U)<2
 \bigr)\,d\vec d=0 ,
\]
the point $\vec d=0$ being Lebesgue null.  The nonnegative random variable
inside the expectation is therefore zero almost surely.  Intersecting over
$R\in\N$, $n\geq2$ and $U\in\mathcal U$ gives one probability-one event on
which $\dim_H((\mathcal A_n)_{\vec d}\cap U)=2$ for Lebesgue-almost every
$\vec d$, simultaneously for all $n$ and $U$; since every nonempty open set
contains some $U\in\mathcal U$, this proves (2).

For (1), fix $U\in\mathcal U$ and $V\in\mathcal V$.  Almost every
$\vec d\in V$ has a fibre of dimension two inside $U$, so by inner
regularity of Lebesgue measure there is a compact $Y\subset V$ of positive
Lebesgue measure consisting of such vectors.  Lemma~\ref{lem:fibre} with
$s=t=2$, applied to $\mathcal A_n\cap(U\times V)$, gives
$\dim_H(\mathcal A_n\cap(U\times V))\geq4$, and the reverse inequality is
trivial.  Every nonempty open subset of $\R^4$ contains such a product
$U\times V$.
\end{proof}

The same construction handles countably many prescribed points.
Intersecting the finite-prefix events would only give ``for every $N$ there
is a translation $a_N$''; using successive finite time intervals produces
one translation that works for every $N$.

\begin{theorem}[countable configurations]\label{prop:countable}
For every fixed countable set $F\subseteq\R^2$, almost surely the
translation set
\[
T_F:=\{a\in\R^2:a+F\subseteq B\}
\]
is dense in $\R^2$.  In particular, the trace contains a translate of $F$.
\end{theorem}

\begin{proof}
If $F=\varnothing$ then $T_F=\R^2$.  Otherwise enumerate
$F=\{x_1,x_2,\ldots\}$, with repetitions if $F$ is finite.  Fix a
nonempty open set $V\subseteq\R^2$, a closed disk $K_0\subset V$ of positive
radius, and put $\tau_0=0$.  Given an $\mathcal F_{\tau_{n-1}}$-measurable
compact $K_{n-1}$ with $\dim_HK_{n-1}>1$, apply
Lemma~\ref{lem:intersection} with $q=1$ after $\tau_{n-1}$ to
$x_n+K_{n-1}$: it gives a finite stopping time $\tau_n>\tau_{n-1}$ such
that
\[
K_n:=\bigl((x_n+K_{n-1})
       \cap W([\tau_{n-1},\tau_n])\bigr)-x_n
\]
is an $\mathcal F_{\tau_n}$-measurable compact subset of $K_{n-1}$ with
$\dim_HK_n>1$.  All countably many stages succeed on one probability-one
event, on which $K_0\supseteq K_1\supseteq\cdots$ are nonempty compact
sets, so $\bigcap_nK_n\ne\varnothing$.  Any $a$ in this intersection
satisfies $a+x_n\in W([\tau_{n-1},\tau_n])\subseteq B$ for every $n$, so
$a\in T_F\cap V$.  Intersecting over a countable base of open sets $V$
proves density.  The intersection of the nested compact sets $K_n$ is
nonempty, but its Hausdorff dimension may be zero; this is why the argument
gives density rather than a dimension bound.
\end{proof}

\begin{remark}
The finite time intervals are what make the compactness argument work: the
nested sets $T_{F_N}=\{a:a+\{x_1,\ldots,x_N\}\subseteq B\}$ need not be
compact, because $B$ is not closed.  Indeed, $B$ is dense by neighbourhood
recurrence \cite[Thm.~3.20]{PM}, while every fixed point is polar at positive
times \cite[Cor.~2.26]{PM}, so Tonelli's theorem gives
$\mathbb E\bigl[\operatorname{Leb}(B\cap[-R,R]^2)\bigr]
=\int_{[-R,R]^2}\mathbb P(x\in B)\,dx=0$ for every $R$; hence
$\operatorname{Leb}(B)=0$ almost surely, and a dense closed subset of
$\R^2$ would be all of $\R^2$.  Each $K_n$ above, in contrast, is defined
through the range over a finite time interval and is compact.
\end{remark}

\subsection{Infinite progressions}\label{subsec:infinite}

Let
\[
\mathcal A_\infty
:=\bigl\{(a,\vec d)\in\R^2\times(\R^2\setminus\{0\}):
          a+j\vec d\in B\text{ for every }j\geq0\bigr\}.
\]
In the path-space notation above,
\[
 \mathfrak A_\infty
 :=\{(\omega,a,\vec d):(a,\vec d)\in\mathcal A_\infty(\omega)\}
 =\bigcap_{j=0}^{\infty}
   \{(\omega,a,\vec d):\vec d\ne0,\ a+j\vec d\in B(\omega)\}
\]
is Borel, so $\mathcal A_\infty(\omega)$ is Borel for every path $\omega$.

\begin{theorem}[infinite progressions]\label{thm:infiniteAP}
For each fixed $\vec d\neq0$, the fibre
\[
(\mathcal A_\infty)_{\vec d}
:=\{a\in\R^2:a+j\vec d\in B\text{ for every }j\geq0\}
\]
is almost surely dense in $\R^2$.  Moreover, there is an event $\Omega_0$
of probability one on which the fibre is dense for all $\vec d$ outside a
Lebesgue-null set that may depend on the path; on $\Omega_0$,
$\mathcal A_\infty$ is dense in $\R^4$ and
\[
2\leq\dim_H\mathcal A_\infty\leq4.
\]
\end{theorem}

\begin{proof}
Fix $\vec d\neq0$.  The countable set $F_{\vec d}:=\{j\vec d:j\geq0\}$ is
deterministic, so by Theorem~\ref{prop:countable} its translation set,
which is $(\mathcal A_\infty)_{\vec d}$, is almost surely dense.

For the second assertion, let $U$ be a rational closed disk and
\[
 G_U:=\{(\omega,\vec d):\vec d\neq0,
       (\mathcal A_\infty(\omega))_{\vec d}\cap U\neq\varnothing\},
\]
the set of pairs $(\omega,\vec d)$ for which the path contains an infinite
progression with gap vector $\vec d$ and starting point in $U$.  It is the
projection of the Borel set
$\{(\omega,\vec d,a):(\omega,a,\vec d)\in\mathfrak A_\infty,\ a\in U\}$,
hence analytic and therefore measurable with respect to the completed
product of Wiener and Lebesgue measure \cite[Thm.~21.10]{Kechris}.  By
the first part, $\mathbb P((\omega,\vec d)\notin G_U)=0$ for every fixed
$\vec d\neq0$, so Tonelli's theorem gives, for every $R\geq1$,
\[
 \mathbb E\Bigl[
  \operatorname{Leb}\{\vec d\in B(0,R)\setminus\{0\}:
                  (\omega,\vec d)\notin G_U\}\Bigr]
 =\int_{B(0,R)\setminus\{0\}}
   \mathbb P((\omega,\vec d)\notin G_U)\,d\vec d=0 ,
\]
and the random Lebesgue measure inside is zero almost surely.  Intersecting
over $R\in\N$ and the countably many $U$ gives $\Omega_0$, on which, for
Lebesgue-almost every $\vec d\neq0$, the fibre meets every rational closed
disk and is therefore dense.

On $\Omega_0$ the projection of $\mathcal A_\infty$ to gap space contains a
set of full Lebesgue measure, hence has dimension two; the projection is
Lipschitz, so $\dim_H\mathcal A_\infty\geq2$, and the upper bound is
trivial.  Finally, a nonempty open subset of $\R^4$ contains a product
$U\times V$ of nonempty open subsets of translation and gap space; almost
every $\vec d\in V$ has a dense fibre, which meets $U$, so
$\mathcal A_\infty$ meets $U\times V$.
\end{proof}

\section{A dimensional threshold for Brownian patterns}
\label{section:Rd}

From dimension three on, patterns on the trace are governed by the number
of effective parameters with which they can be placed: redundant parameters
must be removed, the scale parameter counts like every other parameter, and
a subconfiguration can be overdetermined even when the full parameter count
is favourable.  The criterion below therefore involves the ranks of all
subconfigurations.  The prototype is the theorem of Benjamini and Kozma.

\begin{theorem}[Benjamini--Kozma {\cite[Lemmas~2 and 4]{BK}}]\label{thm:BK}
The trace of Brownian motion in $\R^3$ almost surely contains an arithmetic
progression of length $5$ and almost surely contains none of length $6$.
\end{theorem}

\subsection{Hitting estimates in \texorpdfstring{$\R^d$}{Rd}}
\label{subsec:prelim}

Let $W=(W(t))_{t\ge0}$ be Brownian motion in $\R^d$, $d\ge3$, started at
the origin, let $B:=W([0,\infty))$ be its trace, and put
\[
q:=d-2 .
\]
We keep the probability conventions of Section~\ref{section:plane}.
Brownian motion is invariant under orthogonal maps \cite[Lem.~12.11]{PM}
and transient: $|W(t)|\to\infty$ almost surely \cite[Thm.~3.20]{PM}.
Transience and path continuity make $B$ closed, although unbounded: if
$W(t_n)$ converges in $\R^d$, then $(t_n)$ is bounded, and a convergent
subsequence of times puts the limit in $B$.  Points are \emph{polar}: for
every $x\in\R^d$, $\mathbb P_0\{W(t)=x\text{ for some }t>0\}=0$, by
\cite[Cor.~2.26]{PM}, Brownian scaling and a countable union over integer
time intervals; in particular the trace visits the origin only at time $0$.
The exact ball-hitting formula is
\begin{equation}\label{eq:ballformula}
 \mathbb P_z\bigl(B\cap\overline{B(w,\eps)}\ne\varnothing\bigr)
 =\left(\frac{\eps}{|z-w|}\right)^q,
 \qquad |z-w|>\eps,
\end{equation}
see \cite[Cor.~3.19]{PM}.  Thus $q$ is the hitting codimension of a point:
a labelled subconfiguration $I$ has probabilistic cost $q\,\# I$, to be
compared with the rank $r_I$ counting the parameters available to place it.

\begin{lemma}[single-ball hitting]\label{lem:single}
For $z,w\in\R^d$ and $\eps>0$ set
\[
g_\eps(z,w):=\left(\frac{\eps}{|z-w|+\eps}\right)^q.
\]
Then
\[
g_\eps(z,w)\le
\mathbb P_z\bigl(B\cap B(w,\eps)\ne\varnothing\bigr)
\le 2^qg_\eps(z,w),
\]
and the same comparison holds uniformly when the starting point is moved by
at most $\eps$.
\end{lemma}

\begin{proof}
Let $r=|z-w|$.  If $r<\eps$ the ball is hit at time zero.  If $r\ge\eps$,
the event of hitting $B(w,\eps)$ is the increasing union over
$\rho\uparrow\eps$ of the events of hitting $\overline{B(w,\rho)}$, so by
\eqref{eq:ballformula} it has probability $(\eps/r)^q$; since
$r<r+\eps\le2r$, this lies between $g_\eps(z,w)$ and $2^qg_\eps(z,w)$.  If
$|z-z'|\le\eps$, then $|z-w|+\eps$ and $|z'-w|+\eps$ differ by at most a
factor of two, which gives the uniform version.
\end{proof}

\begin{notation}\label{not:asymp}
$A\lesssim B$ means $A\le CB$ for a constant $C$ independent of the small
scale $\eps$; $A\gtrsim B$ means the reverse inequality, and $A\asymp B$
means both.  Constants may depend on $d$, on the number of points and on a
fixed compact parameter chart, but never on $\eps$.
\end{notation}

We use the Paley--Zygmund inequality \cite[Lem.~3.23]{PM} in the form
\begin{equation}\label{eq:PZ}
\mathbb P(Z>0)\ge\frac{\mathbb E(Z)^2}{\mathbb E(Z^2)}
\qquad(Z\ge0,\ 0<\mathbb E(Z^2)<\infty),
\end{equation}
which is Cauchy--Schwarz applied to
$\mathbb EZ=\mathbb E(Z\mathbf1_{\{Z>0\}})$.

\subsection{Several targets and two nearby copies}\label{subsec:master}

Keeping the labels in disjoint spatial boxes prevents a second
parametrisation of the same configuration, such as the reversal of a
progression, from creating extra multiplicities in the pair count.

\begin{lemma}[multi-ball and paired-copy estimates]\label{lem:master}
Fix $k$ and $a>0$, and let $V_1,\dots,V_k$ be fixed bounded sets such that
\[
 \operatorname{dist}(V_i,0)\ge a,
 \qquad \operatorname{dist}(V_i,V_l)\ge a\quad(i\ne l).
\]
For $x_i\in V_i$, let $H_x$ be the event that all balls $B(x_i,\eps)$ are
hit.  Then, uniformly for $0<\eps<a/10$,
\begin{equation}\label{eq:onecopyhit}
 \mathbb P(H_x)\asymp\eps^{qk},
\end{equation}
and for any second choice $y_i\in V_i$,
\begin{equation}\label{eq:pairedhit}
 \mathbb P(H_x\cap H_y)
 \lesssim
 \eps^{qk}\prod_{i=1}^k g_\eps(x_i,y_i).
\end{equation}
\end{lemma}

\begin{proof}
All constants depend only on $d,k,a$ and the sets $V_1,\ldots,V_k$.  If
$u\in V_i$, $v\in V_l$ with $i\ne l$, and $z\in\overline{B(u,\eps)}$, then
$|z-v|\ge a-\eps$, so \eqref{eq:ballformula} gives
\begin{equation}\label{eq:differentlabeltransition}
 \mathbb P_z\bigl(B\cap\overline{B(v,\eps)}\ne\varnothing\bigr)
 \le \left(\frac{\eps}{a-\eps}\right)^q
 \le C\eps^q,
\end{equation}
also for $z=0$.  Conversely, boundedness of the $V_i$ and
\eqref{eq:ballformula} with radius $\eps/4$ give
\begin{equation}\label{eq:differentlabeltransitionlower}
 \mathbb P_z\bigl(B\cap\overline{B(v,\eps/4)}\ne\varnothing\bigr)
 \ge c\eps^q
\end{equation}
for $z=0$ and for $z$ within $\eps/4$ of a centre carrying a different
label.  Both estimates are uniform in the centres.

\emph{Lower bound in \eqref{eq:onecopyhit}.}  Put $\tau_0=0$ and
\[
 \tau_i:=\inf\{t\ge\tau_{i-1}:
 W(t)\in\overline{B(x_i,\eps/4)}\},
 \qquad 1\le i\le k,
\]
with all later times infinite once one of them is.  By the strong Markov
property and \eqref{eq:differentlabeltransitionlower}, each of the $k$
transitions has conditional probability at least $c\eps^q$, so
$\mathbb P(\tau_k<\infty)\ge c^k\eps^{qk}$; and $\{\tau_k<\infty\}\subseteq H_x$
because the closed balls of radius $\eps/4$ lie inside the open balls of
radius $\eps$.

\emph{Upper bound in \eqref{eq:onecopyhit}.}  On $H_x$ all the closed balls
$\overline{B(x_i,\eps)}$ are hit; they are pairwise disjoint, so their
first-hit order has no ties.  Fix an order and drop the requirement that
later balls are not hit earlier.  Iterating
\eqref{eq:differentlabeltransition} with the strong Markov property bounds
the resulting sequential hitting event by $C^k\eps^{qk}$; summing over the
at most $k!$ orders proves the upper bound.

\emph{Proof of \eqref{eq:pairedhit}.}  Put
\[
 R:=\{i:|x_i-y_i|>2\eps\},\qquad s:=\# R.
\]
For $i\notin R$ we have $g_\eps(x_i,y_i)\ge3^{-q}$, so we may drop the
second target ball $B(y_i,\eps)$ for such labels at the cost of a constant.
The $k+s$ remaining closed targets are pairwise disjoint; fix their
first-hit order.  A transition from the origin to a target, or between
different labels, is bounded by $C\eps^q$ as above.  For an equal-label pair
with $|x_i-y_i|>2\eps$ and $z\in\overline{B(x_i,\eps)}$,
\eqref{eq:ballformula} gives
\[
 \mathbb P_z\bigl(B\cap\overline{B(y_i,\eps)}\ne\varnothing\bigr)
 \le \left(\frac{\eps}{|x_i-y_i|-\eps}\right)^q
 \le 3^q g_\eps(x_i,y_i),
\]
and the same holds with $x_i,y_i$ interchanged.  Suppose the two
equal-label targets are consecutive exactly for the labels in
$J\subseteq R$, and put $c=\#J$.  Among the $k+s$ transitions (counting the
initial one from the origin) there are $c$ equal-label transitions and
$k+s-c$ others, so this order contributes at most
$C\eps^{q(k+s-c)}\prod_{i\in J}g_\eps(x_i,y_i)$.  Since all centres lie in a
fixed bounded set, $g_\eps(x_i,y_i)\ge c_0\eps^q$ for $i\in R\setminus J$,
and hence
\[
 \eps^{q(k+s-c)}\prod_{i\in J}g_\eps(x_i,y_i)
 =\eps^{qk}\prod_{i\in J}g_\eps(x_i,y_i)
   \prod_{i\in R\setminus J}\eps^q
 \le C\eps^{qk}\prod_{i\in R}g_\eps(x_i,y_i).
\]
Summing over the at most $(2k)!$ orders and restoring the factors with
$i\notin R$, each bounded below by $3^{-q}$, proves \eqref{eq:pairedhit}.
\end{proof}

\subsection{Pattern charts and effective dimension}\label{subsec:family}

The next definitions record the number $D$ of independent parameters of a
placed pattern and, for every set of labels $I$, the dimension $r_I$ of the
possible placements of that subconfiguration.  The probabilistic argument
depends on the family only through these ranks.

Write $\operatorname{Conf}_k(\R^d)$ for the ordered configurations of $k$
distinct points in $\R^d$.  Unordered families are treated through local
lifts to this space; different lifts are harmless because the atlas is
countable.  A family is \emph{scale invariant} if it contains
$(x_1,\ldots,x_k)$ exactly when it contains
$(\lambda x_1,\ldots,\lambda x_k)$, for every $\lambda>0$.  A configuration
is \emph{nonanchored} if none of its points is the origin.  A $C^1$ map has
\emph{constant rank} if its differential has the same rank at every point
of its domain; it is a \emph{submersion} if its differential is surjective
at every point, and an \emph{embedding} if it is injective, its
differential is injective at every point, and it is a homeomorphism onto
its image.  By the constant-rank theorem, a map of constant rank $r$
sends a small neighbourhood of any point onto an $r$-dimensional $C^1$
submanifold.

\begin{definition}[normalised pattern chart]\label{def:family}
Let $U\subset\R^D$ be open, let
\[
 \Phi=(\phi_1,\dots,\phi_k):U\longrightarrow
 \operatorname{Conf}_k(\R^d)
\]
be $C^1$, and let $K\subset U$ be a compact parameter cube with nonempty
interior.  We call $(K,\Phi)$ a \emph{normalised rank chart} if there is
$a>0$ such that \ref{N1}--\ref{N2} hold, and a \emph{normalised
positive-excess chart} if \ref{N3} holds as well.
\begin{enumerate}[label=\normalfont(N\arabic*)]
\item\label{N1} \emph{Geometry and labels:} for all $\omega,\omega'\in K$,
\[
 a\le |\phi_i(\omega)|\le a^{-1},\qquad
 |\phi_i(\omega)-\phi_l(\omega')|\ge a\quad(i\ne l),
\]
and the image of each single label has diameter at most $a/10$.
\item\label{N2} \emph{Effective parametrisation:} $\Phi$ is an embedding on
a neighbourhood of $K$, every $\phi_i$ is a submersion, and every submap
$\Phi_I=(\phi_i)_{i\in I}$ has constant rank $r_I$ there.  By
Lemma~\ref{lem:measurebox}, if $\pi$ is a permutation,
$I_j=\{\pi(1),\dots,\pi(j)\}$, $0<t_1\le\cdots\le t_k\le1$ and
$\omega\in K$, then
\begin{equation}\label{eq:flagbox}
\operatorname{Leb}\bigl\{\omega'\in K:
 |\phi_{\pi(j)}(\omega')-\phi_{\pi(j)}(\omega)|\le t_j,
 \ 1\le j\le k\bigr\}
 \lesssim\prod_{j=1}^k t_j^{\,r_{I_j}-r_{I_{j-1}}},
 \qquad r_\varnothing=0,
\end{equation}
with a constant depending only on the chart and on the maps $\phi_i$ on a
fixed neighbourhood of $K$, uniform in $\omega$, in the radii $t_j$ and in
the finitely many permutations $\pi$.
\item\label{N3} \emph{Proper subconfigurations:}
\[
 \Delta_I:=r_I-q\,\# I>0
 \qquad(\varnothing\ne I\subsetneq\{1,\dots,k\}).
\]
\end{enumerate}
Since $\Phi$ is an embedding, $r_{\{1,\dots,k\}}=D$: this is the number of
nonredundant parameters of a placed copy, with translation, rotation and
dilation counted whenever they are free.  At resolution $\eps$ a
$D$-dimensional parameter set has about $\eps^{-D}$ cells, while hitting
the $k$ target balls costs a factor $\eps^{qk}$; this gives the
\emph{full excess}
\begin{equation}\label{eq:Delta}
 \Delta:=D-k(d-2)=D-qk.
\end{equation}
\end{definition}

The left-hand side of \eqref{eq:flagbox} is the volume of parameters
$\omega'$ whose first selected label lies within $t_1$ of its position at
$\omega$, whose first two selected labels satisfy their constraints at
scale $t_2$, and so on; the nested label sets $I_1\subset\cdots\subset I_k$
form a flag, whence the name \emph{flag estimate}.
Lemma~\ref{lem:measurebox} derives it from the constant-rank hypothesis on
a neighbourhood of the compact chart.

\begin{definition}[rank-regular scale atlas]\label{def:atlas}
A nonempty scale-invariant family $\mathcal F$ of unordered $k$-point
configurations
has a \emph{rank-regular scale atlas} if it is covered, after local
labelling, by countably many $C^1$ embeddings
\[
 \Psi_\nu:U_\nu\subset\R^{D_\nu}\longrightarrow
 \operatorname{Conf}_k(\R^d)
\]
such that every coordinate is a submersion and every submap has constant
rank on a neighbourhood of each compact subchart.  These global charts may
pass through configurations containing the origin; every nonanchored
parameter can be refined to a normalised rank chart.  Dilation is one of the
$D_\nu$ effective parameters, not an external factor, and countability comes
from second countability and a compact exhaustion, not from quotienting out
the scale.  The atlas is \emph{fixed-rank} if all its charts have the same
$D$ and, after the corresponding permutations of local labels, the same
ranks $r_I$; the collection of inequalities over all $I$ is unchanged by
such permutations.
\end{definition}

\subsection{Ranks and resonance}\label{subsec:resonance}

The occurrence argument of Section~\ref{subsec:subcritical} counts the
parameters $\omega$ in an $\eps$-net of $K$ whose $k$ target balls
$B(\phi_i(\omega),\eps)$ are all hit.  The expected count is of order
$\eps^{-\Delta}$, but an expectation alone does not show that the count is
positive with probability bounded below, since it could be large on an
event of small probability.  The Paley--Zygmund inequality \eqref{eq:PZ}
also needs the second moment to be of order $\eps^{-2\Delta}$, that is,
the hitting events of two candidates must be only weakly correlated on
average.  By Lemma~\ref{lem:master}, this correlation is measured by the
product of the factors $g_\eps(\phi_i(\omega),\phi_i(\omega'))$: a label
whose two targets are within $\eps$ of each other costs nothing extra,
while one whose targets are at distance $\rho$ costs $(\eps/\rho)^q$.  The
flag estimate \eqref{eq:flagbox} bounds the number of net points whose
first $j$ selected labels lie within prescribed distances of those of
$\omega$, and positivity of every prefix excess $\Delta_{I_j}$ makes the
resulting sum over scales a convergent geometric series
(Lemma~\ref{lem:resonance}).  This gives a copy with positive probability
(Theorem~\ref{thm:localtrichotomy}); probability one is a separate step
that uses scale invariance (Lemma~\ref{lem:scaleergodic}).

For linear maps $A_i:\R^D\to\R^d$, write $A_I=(A_i)_{i\in I}$.  The
maximal-rank condition
\begin{equation}\label{eq:transversality}
 \operatorname{rank}(A_I)=\min(d\,\# I,D)
 \qquad(\varnothing\ne I\subseteq\{1,\dots,k\})
\end{equation}
says that every set of labels has the largest rank permitted by its domain
and target; it is stronger than what the criterion needs.

\begin{lemma}[flag boxes from constant ranks]\label{lem:measurebox}
Let $U\subset\R^D$ be open, let $K\subset U$ be compact, let
$f_1,\dots,f_k:U\to\R^d$ be $C^1$, and fix a permutation $\pi$.  Suppose
every nested map $F_j=(f_{\pi(1)},\dots,f_{\pi(j)})$ has constant rank
$\rho_j$ on a neighbourhood of $K$.  Then, uniformly in the centres $y_j$
and in $0<t_1\le\cdots\le t_k\le1$,
\[
 \operatorname{Leb}\{x\in K:|f_{\pi(j)}(x)-y_j|\le t_j,
 \ 1\le j\le k\}
 \lesssim\prod_{j=1}^k t_j^{\rho_j-\rho_{j-1}},
 \qquad \rho_0=0,
\]
with a constant depending only on $K$, on the maps $f_i$ on a fixed
neighbourhood of $K$, and on $\pi$.
\end{lemma}

\begin{proof}
Put $a_j=\rho_j-\rho_{j-1}\ge0$ and fix $x_0\in K$.  The row spaces
$\mathcal R_j:=\operatorname{row}DF_j(x_0)\subseteq(\R^D)^*$,
$\mathcal R_0=\{0\}$, form a flag
$\mathcal R_0\subseteq\mathcal R_1\subseteq\cdots\subseteq\mathcal R_k$
with $\dim\mathcal R_j=\rho_j$, and $\mathcal R_j$ is spanned by
$\mathcal R_{j-1}$ and the rows of $Df_{\pi(j)}(x_0)$.  Select $a_j$
linear combinations of the latter rows whose classes form a basis of
$\mathcal R_j/\mathcal R_{j-1}$; equivalently, choose a linear map
$L_j:\R^d\to\R^{a_j}$ such that the rows of $D(L_jf_{\pi(j)})(x_0)$
complement $\mathcal R_{j-1}$ inside $\mathcal R_j$ (this block is omitted
when $a_j=0$).  All selected rows form a basis of $\mathcal R_k$; complete
it to a basis of $(\R^D)^*$ by the rows of a linear map
$h:\R^D\to\R^{D-\rho_k}$.  Since $\sum_ja_j+(D-\rho_k)=D$, the map
\[
 G=(L_1f_{\pi(1)},\ldots,L_kf_{\pi(k)},h):U\longrightarrow\R^D
\]
has invertible derivative at $x_0$, and the inverse function theorem gives
a precompact neighbourhood $P_{x_0}$ on which $G$ is one-to-one and
$|\det DG|$ is bounded below by a positive constant.  Let
\[
 E:=\{x\in K:|f_{\pi(j)}(x)-y_j|\le t_j,\ 1\le j\le k\}.
\]
For $x\in E\cap P_{x_0}$ we have
$|L_jf_{\pi(j)}(x)-L_jy_j|\le\|L_j\|t_j$, so $G(E\cap P_{x_0})$ is
contained in $\prod_{j}B_{\R^{a_j}}(L_jy_j,\|L_j\|t_j)\times H_{x_0}$,
where $H_{x_0}$ is a fixed bounded set containing the possible
$h$-coordinates.  The change of variables formula gives
$\operatorname{Leb}(E\cap P_{x_0})\le C_{x_0}\prod_jt_j^{a_j}$, with a
constant independent of the centres (translating the balls does not change
their volumes) and of the radii.  Since the ranks are constant on a
neighbourhood of $K$, compactness provides finitely many patches $P_{x_0}$
covering $K$; summing their estimates proves the lemma.  The constant-rank
hypothesis is what produces the nondegenerate coordinate map $G$; a
Lipschitz bound alone would give no lower Jacobian control.  The ordering
of the radii is not used here; it matters in the flag decomposition below.
\end{proof}

\begin{corollary}[rank check]\label{cor:Timpliesreg}
If the full map of a $C^1$ family is an embedding, its individual maps have
rank $d$, and its submaps have constant ranks on a neighbourhood of the
compact chart, then \ref{N2}, including \eqref{eq:flagbox}, holds; if
every proper subfamily has positive excess, \ref{N3} holds as well.  For an
affine family this is a finite matrix-rank computation.  The maximal-rank
condition \eqref{eq:transversality} is a convenient way to compute the
ranks, but the excess inequalities must still be checked.
\end{corollary}

A set $\Lambda_\eps\subset K$ is a \emph{maximal $\eps$-separated net} if
its points are at mutual distance at least $\eps$ and no point of $K$ can
be added; then every point of $K$ is within $\eps$ of $\Lambda_\eps$, and
$\#\Lambda_\eps\asymp\eps^{-D}$ because $K$ is a $D$-dimensional cube.  Put
\[
 Q_\eps(\omega,\omega'):=
 \prod_{i=1}^k g_\eps\bigl(\phi_i(\omega),\phi_i(\omega')\bigr).
\]

\begin{lemma}[resonance bound]\label{lem:resonance}
Let $(K,\Phi)$ satisfy \ref{N1}--\ref{N3} and suppose $\Delta>0$.  Then,
for all sufficiently small $\eps>0$, uniformly for $\omega\in\Lambda_\eps$,
\begin{equation}\label{eq:resonancesum}
 \sum_{\omega'\in\Lambda_\eps}Q_\eps(\omega,\omega')
 \lesssim\eps^{-\Delta}.
\end{equation}
\end{lemma}

\begin{proof}
Choose a compact neighbourhood $K^+\subset U$ of $K$ on which all submap
ranks remain constant; for small $\eps$ every $C\eps$-thickening used
below lies in $K^+$, and Lemma~\ref{lem:measurebox} applied to $K^+$ gives
the flag estimate for the thickened cells.  An $\eps$-separated net has at
most $C\eps^{-D}\operatorname{Leb}(A^{C\eps})$ points in a set $A$, since
the balls of radius $\eps/3$ about those points are disjoint and lie in
$A^{\eps/3}$; thickening a dyadic cell enlarges each flag tolerance by only
$O(\eps)$, which is absorbed at the lowest dyadic scale.

Decompose dyadically the $k$ quantities
$|\phi_i(\omega')-\phi_i(\omega)|+\eps$, and choose $L$ with
$2^{-L}\asymp\eps$.  On each dyadic cell, order the labels by a permutation
$\pi$ so that
\[
 |\phi_{\pi(j)}(\omega')-\phi_{\pi(j)}(\omega)|+\eps
 \asymp 2^{-n_j},
 \qquad L\ge n_1\ge\cdots\ge n_k\ge0 ;
\]
for small $\eps$, \ref{N1} keeps these same-label distances below $1$, so
the radii lie in the range of \eqref{eq:flagbox}.  Put
$I_j=\{\pi(1),\ldots,\pi(j)\}$, $a_j=r_{I_j}-r_{I_{j-1}}$, $n_{k+1}=0$
and $g_j=n_j-n_{j+1}\ge0$.  The net count contributes $\eps^{-D}$, the
kernels contribute $\eps^{qk}2^{q\sum_j n_j}$, and \eqref{eq:flagbox}
contributes $2^{-\sum_j n_ja_j}$.  Abel summation gives
\[
 \sum_{j=1}^k n_j(a_j-q)
 =\sum_{j=1}^k g_j\sum_{h=1}^j(a_h-q)
 =\sum_{j=1}^k g_j\Delta_{I_j},
\]
so, for a fixed permutation, the sum over dyadic cells is at most
\[
 \eps^{qk-D}
 \sum_{g_1,\dots,g_k\ge0}
 2^{-\sum_{j=1}^k g_j\Delta_{I_j}},
\]
where extending the gaps to all nonnegative integers only enlarges the sum.
Every exponent is positive, by \ref{N3} for a proper prefix and by
$\Delta>0$ for the last one, so the product of geometric series is bounded,
and $\eps^{qk-D}=\eps^{-\Delta}$.  Summing over the finitely many
permutations proves \eqref{eq:resonancesum}.
\end{proof}

\subsection{Positive excess: the second-moment argument}
\label{subsec:subcritical}

For $\omega\in\Lambda_\eps$, let
\[
 H_\omega:=\bigcap_{i=1}^k
 \{B\cap B(\phi_i(\omega),\eps)\ne\varnothing\},
 \qquad
 X_\eps:=\sum_{\omega\in\Lambda_\eps}\mathbf1_{H_\omega}.
\]
Lemma~\ref{lem:master} and \ref{N1} give
\begin{equation}\label{eq:PHom}
 \mathbb P(H_\omega)\asymp\eps^{qk}
\end{equation}
and, for two parameters in the same labelled chart,
\begin{equation}\label{eq:pairom}
 \mathbb P(H_\omega\cap H_{\omega'})
 \lesssim \mathbb P(H_\omega)Q_\eps(\omega,\omega').
\end{equation}
Since $\#\Lambda_\eps\asymp\eps^{-D}$ and
$\mathbb EX_\eps=\sum_{\omega}\mathbb P(H_\omega)$,
\begin{equation}\label{eq:firstmomsub}
 \mathbb E(X_\eps)\asymp\eps^{-\Delta},
\end{equation}
and if $\Delta>0$, Lemma~\ref{lem:resonance} gives
\begin{equation}\label{eq:secondmomsub}
 \begin{split}
 \mathbb E(X_\eps^2)
 &=\sum_{\omega,\omega'\in\Lambda_\eps}
   \mathbb P(H_\omega\cap H_{\omega'})
 \lesssim\sum_{\omega\in\Lambda_\eps}\mathbb P(H_\omega)
   \sum_{\omega'\in\Lambda_\eps}Q_\eps(\omega,\omega')\\
 &\lesssim \mathbb E(X_\eps)\eps^{-\Delta}
 \asymp \mathbb E(X_\eps)^2.
 \end{split}
\end{equation}

For a compact parameter cube $K$, define the decreasing events
\[
 E_r(K):=\left\{\exists\omega\in K:
 \operatorname{dist}(\phi_i(\omega),B)\le r
 \text{ for every }i\right\}.
\]
They are measurable: on the probability-one event on which $B$ is closed,
transience and compactness of $\Phi(K)$ give
\[
 E_r(K)=\bigcup_{T\in\N}
 \left\{
 \min_{\substack{\omega\in K\\t_1,\ldots,t_k\in[0,T]}}
 \max_{1\le i\le k}|W(t_i)-\phi_i(\omega)|\le r
 \right\},
\]
a countable union of Borel events (the minimum is over a compact set), and
completion takes care of the null event.  Moreover, almost surely,
\begin{equation}\label{eq:nestedexact}
 \bigcap_{n\ge1}E_{2^{-n}}(K)
 =\{B\text{ contains }\Phi(\omega)\text{ for some }\omega\in K\}:
\end{equation}
if $\omega_n\in K$ witnesses $E_{2^{-n}}(K)$, a subsequence converges to
some $\omega\in K$, and points $b_{i,n}\in B$ within $2^{-n}$ of
$\phi_i(\omega_n)$ converge along it to $\phi_i(\omega)$, which lies in $B$
because $B$ is closed; the reverse inclusion is immediate.

\begin{theorem}[local occurrence]\label{thm:localtrichotomy}
If a normalised positive-excess chart satisfies $\Delta>0$, then the trace
contains a copy with parameter in $K$ with probability bounded below by a
positive constant.
\end{theorem}

\begin{proof}
By \eqref{eq:PZ}, \eqref{eq:firstmomsub} and \eqref{eq:secondmomsub},
$\mathbb P(X_\eps>0)\gtrsim1$ uniformly in $\eps$.  Since
$\{X_\eps>0\}\subseteq E_\eps(K)$, all the events $E_{2^{-n}}(K)$ have
probability bounded below by the same positive constant, and continuity of
probability from above gives
$\mathbb P\bigl(\bigcap_{n}E_{2^{-n}}(K)\bigr)
=\lim_{n}\mathbb P(E_{2^{-n}}(K))>0$.  Now apply \eqref{eq:nestedexact}.
The nets $\Lambda_\eps$ need not be nested; they only provide the uniform
lower bound.
\end{proof}

The first moment alone gives a one-sided dimension bound; it does not use
the resonance estimate and is not an exact dimension statement.  For a
nonempty bounded set $A$ let $N(A,r)$ be the least number of balls of
radius $r$ needed to cover $A$, and let
\[
 \overline{\dim}_M A
 :=\limsup_{r\downarrow0}\frac{\log N(A,r)}{-\log r}
\]
be its upper Minkowski dimension, with $\overline{\dim}_M\varnothing=0$.

\begin{proposition}[upper dimension of realised parameters]
\label{prop:realisedupper}
For a normalised rank chart with $\Delta\ge0$, let
\[
 \mathcal R_K:=\{\omega\in K:\phi_i(\omega)\in B
                       \text{ for every }1\le i\le k\}.
\]
Almost surely, $\dim_H\mathcal R_K\le\overline{\dim}_{M}\mathcal R_K\le\Delta$.
\end{proposition}

\begin{proof}
Let $\Lambda_\eps$ be a maximal $\eps$-separated net of $K$, let $C$ be a
common Lipschitz constant of the $\phi_i$ on $K$, and put $C_0=C+1$.
Almost surely $\mathcal R_K$ is compact, because $K$ and $B$ are closed.
If $\omega\in\mathcal R_K$ and $\omega'\in\Lambda_\eps$ is within $\eps$ of
$\omega$, then $B$ meets every ball $B(\phi_i(\omega'),C_0\eps)$.  Let
$Y_\eps$ be the number of net points with this property; the corresponding
balls cover $\mathcal R_K$, and \eqref{eq:onecopyhit} with $C_0\eps$ in
place of $\eps$ gives
$\mathbb EY_\eps\lesssim\eps^{-D}\eps^{qk}=\eps^{-\Delta}$.  For
$\eta>0$, Markov's inequality gives
$\mathbb P\{Y_{2^{-n}}>2^{n(\Delta+\eta)}\}\lesssim2^{-n\eta}$, which is
summable, so by Borel--Cantelli, almost surely $\mathcal R_K$ is covered by
at most $2^{n(\Delta+\eta)}$ parameter balls of radius $O(2^{-n})$ for all
large $n$.  Intersecting over rational $\eta>0$ proves the Minkowski
bound, and the Hausdorff bound follows.
\end{proof}

The upgrade from positive probability to probability one uses the
ergodicity of Brownian scaling, in the standard
Lamperti--Ornstein--Uhlenbeck formulation; compare \cite[p.~3]{CKM}.  We
include the short proof because we use the precise action on one-sided
Wiener space.  A probability-preserving map $T$ is \emph{ergodic} if every
event with $T^{-1}A=A$ has probability zero or one, and \emph{mixing} if
$\mathbb P(A\cap T^{-n}C)\to\mathbb P(A)\mathbb P(C)$ for all events
$A,C$; mixing implies ergodicity \cite[Chapter~1]{Walters}.

\begin{lemma}[ergodicity of Brownian scaling]\label{lem:scaleergodic}
For every $\lambda>1$, the map
\[
 (S_\lambda W)(t):=\lambda^{-1}W(\lambda^2t)
\]
is measure preserving and ergodic on one-sided Wiener space.  Hence every
event determined by the trace through a property invariant under spatial
dilation has probability $0$ or $1$.
\end{lemma}

\begin{proof}
$S_\lambda$ is measure preserving by Brownian scaling \cite[Lem.~1.7]{PM}.
The Lamperti transform $Y(s)=e^{-s/2}W(e^s)$, $s\in\R$ (see also
\cite[Rem.~1.10]{PM}), is a measurable shift-equivariant isomorphism onto
its image, with inverse $W(t)=t^{1/2}Y(\log t)$ for $t>0$ and $W(0)=0$; it
conjugates $S_\lambda$ to translation by $2\log\lambda$.  The process $Y$
is stationary Gaussian with covariance
$\mathbb E(Y(s)Y(t)^*)=e^{-|t-s|/2}I_d$.  For two finite vectors of sampled
values of $Y$ whose sampling times are separated by $h$, every
cross-covariance tends to zero as $h\to\infty$; after repeated time
coordinates are removed, the covariance matrices converge to a
block-diagonal positive-definite matrix, and the Gaussian densities converge
in $L^1$ to the product density.  Thus cylinder events are mixing.  The
cylinder algebra generates the Borel sigma-field of path space and is dense
for the metric $d(A,C)=\mathbb P(A\mathbin{\triangle}C)$; since
translation is measure preserving, approximating either event changes the
mixing identity uniformly in the shift, so the identity extends to all
measurable events.  Hence the shift, and with it $S_\lambda$, is mixing and
therefore ergodic.  Finally, the trace of $S_\lambda W$ is
$\lambda^{-1}B$, so a dilation-invariant property of the trace defines an
$S_\lambda$-invariant event.
\end{proof}

Only one dilation is used, so the probability-one upgrade remains valid for
families invariant merely under the discrete group $\{\lambda^n:n\in\Z\}$
for some fixed $\lambda>1$; full scale invariance is assumed in the main
statement because it is natural for the families considered here.

\subsection{Product capacity: negative and critical excess}
\label{subsec:nonpositive}

Nonoccurrence, including the critical case, is handled by classical
multiparameter potential theory: a product capacity detects whether $s$
independent Brownian paths can realise a point of a configuration set,
each path at its own time, and its energy test remains decisive at the
critical dimension.

Let $\zeta$ be an exponential time of mean one, independent of the path,
and let
\[
 u_1(z):=\int_0^\infty e^{-t}(2\pi t)^{-d/2}
             e^{-|z|^2/(2t)}\,dt
\]
be the $1$-resolvent density of Brownian motion: for every nonnegative Borel
function $f$, Tonelli's theorem and the Gaussian transition density give
\[
 \mathbb E_x\int_0^\zeta f(W(t))\,dt
 =\int_{\R^d}u_1(y-x)f(y)\,dy ,
\]
so $u_1$ is the Green density of Brownian motion killed at time $\zeta$.
Near the origin $u_1(z)$ is comparable to $|z|^{2-d}$, and only this local
comparison is used.  For $x,y\in(\R^d)^s$ set
\[
 \mathcal K_s(x,y):=\prod_{i=1}^s u_1(x_i-y_i),
 \qquad
 \operatorname{Cap}_s(E):=
 \left[
 \inf_{\mu\in\mathcal P(E)}
 \iint\mathcal K_s(x,y)\,d\mu(x)d\mu(y)
 \right]^{-1},
\]
with $\infty^{-1}=0$, where $\mathcal P(E)$ is the set of Borel probability
measures supported on $E$.  Thus $\operatorname{Cap}_s(E)=0$ means that
every probability measure on $E$ has infinite $\mathcal K_s$-energy.

\begin{lemma}[product Brownian capacity]\label{lem:productcapacity}
Let $E\subset(\R^d)^s$ be Borel.  Let the initial vector
$(W^1(0),\dots,W^s(0))$ have a law absolutely continuous with respect to
Lebesgue measure and, conditionally on it, let
$W^i(t)=W^i(0)+\widetilde W^i(t)$, where $\widetilde W^1,\dots,\widetilde W^s$
are independent standard Brownian motions, independent of the initial
vector.  If $\operatorname{Cap}_s(E)=0$, then
\[
 \mathbb P\{(W^1(t_1),\dots,W^s(t_s))\in E
 \text{ for some }t_1,\dots,t_s>0\}=0.
\]
\end{lemma}

\begin{proof}
Kill the $s$ Brownian motions at independent exponential times of mean one,
and let $h(x_1,\ldots,x_s)$ be the probability, starting from
$(x_1,\ldots,x_s)$, that the killed paths satisfy
$(W^1(t_1),\ldots,W^s(t_s))\in E$ for some $t_1,\ldots,t_s>0$.
Exponentially killed Brownian motion is a transient symmetric Borel right
process, self-dual with respect to Lebesgue measure, and satisfies the
standing hypotheses of Fitzsimmons and Salisbury \cite{FS}.  Let
$U^if(x):=\mathbb E_x\int_0^{\zeta_i}f(W^i(t))\,dt$ and let $\nu_iU^i$ be
the measure $A\mapsto\int U^i\mathbf1_A\,d\nu_i$.  If $\nu_1,\ldots,\nu_s$
are probability measures with $\nu_iU^i\le\operatorname{Leb}$ in the order
of measures, that is, $\nu_iU^i$ has a Lebesgue density at most one, then
\cite[Lem.~3.2(a), Thm.~4.16(i) and proof of Thm.~5.1]{FS} give
\begin{equation}\label{eq:FShitting}
 \int h(x_1,\ldots,x_s)\,d\nu_1(x_1)\cdots d\nu_s(x_s)
 \le2^{3s}\operatorname{Cap}_s(E);
\end{equation}
indeed, the symmetrised kernels in \cite[(4.3)]{FS} all equal
$\mathcal K_s$ because killed Brownian motion is symmetric, so their energy
is the one defining $\operatorname{Cap}_s$.  This is the only input from
multiparameter Markov-process theory.

Let $Q_i$ be half-open unit cubes and let $\nu_i$ be Lebesgue measure
restricted to $Q_i$.  The Green-density identity gives
\[
 \frac{d(\nu_iU^i)}{d\operatorname{Leb}}(y)
 =\int_{Q_i}u_1(x-y)\,dx
 \le\int_{\R^d}u_1(x-y)\,dx=1,
\]
so the $\nu_i$ are admissible in \eqref{eq:FShitting}, and
$\operatorname{Cap}_s(E)=0$ gives $\int_{Q_1\times\cdots\times Q_s}h=0$.
Products of half-open integer unit cubes partition $(\R^d)^s$, so $h=0$
Lebesgue-almost everywhere; integrating against the density of the initial
vector proves the assertion for the killed processes.  Finally, for
$T<\infty$ the killing times all exceed $T$ with probability $e^{-sT}$,
independently of the paths, so a positive hitting probability for the
unkilled processes with $(t_1,\ldots,t_s)\in(0,T]^s$ would give one for the
killed processes.  Take the union over $T\in\N$.
\end{proof}

\begin{lemma}[zero capacity of critical manifolds]\label{lem:polar}
Let $M\subset(\R^d)^s$ be a countable union of compact sets $M_n$, each
contained in a precompact coordinate patch of a $C^1$ embedded manifold of
dimension at most $qs$.  Then $\operatorname{Cap}_s(M)=0$; consequently, $M$
is not hit under any initial law covered by
Lemma~\ref{lem:productcapacity}.
\end{lemma}

\begin{proof}
On a bounded set, the substitution $t=|z|^2v$ in the definition of $u_1$,
with $v$ restricted to a compact subinterval of $(0,\infty)$, gives
$u_1(z)\gtrsim|z|^{-q}$; since $|x_i-y_i|\le|x-y|$, it follows that
$\mathcal K_s(x,y)\gtrsim|x-y|^{-qs}$.  Consider one patch of dimension
$r\le qs$.  It is covered by $N\le C\rho^{-r}$ balls of radius $\rho$;
assigning every point to the first ball containing it gives a Borel
partition into $N$ sets of diameter at most $2\rho$.  For a probability
measure $\mu$ on the patch, $(\mu\times\mu)\{|x-y|\le2\rho\}$ is at least
the sum of the squared masses of the cells, hence at least
$1/N\gtrsim\rho^r$ by Cauchy--Schwarz, and the layer-cake formula gives
\[
 \iint |x-y|^{-qs}\,d\mu(x)d\mu(y)
 \gtrsim\int_0^1\rho^{r-qs-1}\,d\rho=\infty
 \qquad(r\le qs).
\]
Thus every probability measure on one patch has infinite
$\mathcal K_s$-energy.  Disjointify $M=\bigcup_nM_n$ into Borel sets
$F_n\subset M_n$; a probability measure on $M$ gives positive mass to some
$F_n$, its normalised restriction to $F_n$ has infinite energy, and
positivity of the kernel makes the energy of the original measure infinite.
Hence $\operatorname{Cap}_s(M)=0$.
\end{proof}

\begin{lemma}[polar submanifolds for one trace]\label{lem:onepathcapacity}
Let $M\subset\operatorname{Conf}_s(\R^d)$ be a countable union of compact
$C^1$ coordinate patches of dimension at most $qs$.  Almost surely there
are no distinct positive times $t_1,\dots,t_s$ such that
$(W(t_1),\dots,W(t_s))\in M$.
\end{lemma}

\begin{proof}
Distinct spatial coordinates force the hitting times to be distinct.  Cover
the tuples of distinct positive times by countably many products
$J_1\times\cdots\times J_s$ of open intervals with rational endpoints and
pairwise disjoint closures, cover each $J_i$ by countably many compact
intervals $[a_i,b_i]\subset J_i$ whose interiors cover $J_i$, and exhaust
each $(a_i,b_i)$ by the intervals $[a_i+\delta,b_i-\delta]$; it suffices to
treat one resulting product.

A \emph{Brownian bridge} from $x$ at time $a$ to $y$ at time $b$ is the
conditional law of the path on $[a,b]$ given these endpoint values; with
$T=b-a$ and $\widehat W$ a standard Brownian motion started at the origin,
it is the law of
\[
 X(t)=x+\frac{t-a}{T}(y-x)+\widehat W(t-a)
      -\frac{t-a}{T}\widehat W(T),
 \qquad a\le t\le b,
\]
see \cite[Ex.~1.5(a),(b) and Ex.~2.11]{PM}.  By the Markov property and
the independence of increments over disjoint intervals, conditionally on
all values $W(a_i)$ and $W(b_i)$ the path pieces on the intervals
$[a_i,b_i]$ are independent Brownian bridges.  Write
$p_t(x,y)=(2\pi t)^{-d/2}\exp\{-|x-y|^2/(2t)\}$ and choose
$0<\eta<\delta<T/2$.  The value $Z$ of the bridge at time $a+\eta$ has
density
\[
 z\longmapsto\frac{p_\eta(x,z)\,p_{T-\eta}(z,y)}{p_T(x,y)},
\]
and, conditionally on $Z=z$, the law of the bridge on $[a+\eta,b-\delta]$
is absolutely continuous with respect to that of Brownian motion started
at $z$, with Radon--Nikodym derivative
\[
 \frac{p_\delta(W(b-\delta),y)}{p_{T-\eta}(z,y)}
\]
on the corresponding sigma-field.  All denominators are strictly positive.
The interior interval is needed because the bridge is pinned at time $b$,
where it is singular with respect to Brownian motion, while the initial
time $a+\eta$ gives the starting vector a Lebesgue density, as
Lemma~\ref{lem:productcapacity} requires; every time tested in
$[a+\delta,b-\delta]$ is at least $\delta-\eta>0$ after the comparison
Brownian motion starts, so hitting is tested at positive times.  Thus, on
the smaller compact intervals, the joint
starting vector has a Lebesgue density and, conditionally on it, the
independent bridge pieces are absolutely continuous with respect to
independent Brownian motions at strictly positive times, the joint
Radon--Nikodym derivative being the product of the displayed finite positive
factors.  Every null event for independent Brownian motions is therefore
null for the bridge pieces, and Lemmas~\ref{lem:productcapacity} and
\ref{lem:polar} make the hitting probability zero.  Integrating over the
endpoint data, exhausting the intervals from within and taking the countable
union proves the claim.
\end{proof}

The argument includes the critical case $r=qs$, where the energy diverges
logarithmically; this is where a first-moment argument alone would not
suffice.

\subsection{The global trichotomy and subconfigurations}\label{subsec:global}

All occurrence events below are measurable: for a compact parameter set $K$
in one chart and $T<\infty$, the event that a copy is realised by times in
$[0,T]$ is the zero set of
\[
 W\longmapsto
 \min_{\substack{\omega\in K\\t_1,\ldots,t_k\in[0,T]}}
 \max_{1\le i\le k}|W(t_i)-\psi_i(\omega)|,
\]
a continuous function on path space with the topology of uniform
convergence on $[0,T]$ (the minimum is over a compact set), and countable
exhaustions of each chart, of the time axis and of the atlas make the full
occurrence event Borel.

\begin{proposition}[subconfiguration obstruction]\label{prop:obstruction}
Let $\mathcal F$ be covered by a countable rank-regular $C^1$ atlas as in
Definition~\ref{def:atlas}, without any assumption of scale invariance.  If
on one of its charts a nonempty submap $\Psi_I$ has rank
\[
 r_I\le(d-2)\,\# I,
\]
then almost surely the trace contains no configuration from that chart.  In
particular, if every chart has such a subconfiguration, almost surely the
trace contains no member of $\mathcal F$.
\end{proposition}

\begin{proof}
By the constant-rank theorem, the image of $\Psi_I$ is locally a countable
union of $r_I$-dimensional $C^1$ submanifolds.  On the part where none of
its coordinates is the origin, Lemma~\ref{lem:onepathcapacity} rules out a
hit.

It remains to treat configurations containing the starting point.  Since
points are polar, the trace visits the origin only at time $0$, and by
distinctness at most one label $j\in I$ can be at the origin.  Note that
$\# I\ge2$, since $r_{\{i\}}=d>d-2$.  As $\psi_j$ is a submersion,
$Z_j=\{\omega:\psi_j(\omega)=0\}$ is a $C^1$ submanifold.  With
$A=D\psi_j$, which is onto, and $C=D\Psi_{I\setminus\{j\}}$, the identity
\[
 \operatorname{rank}(A,C)
 =\operatorname{rank}A+\operatorname{rank}(C|_{\ker A}),
\]
together with $\ker A=TZ_j$, gives
\[
 \operatorname{rank}
 D\!\left(\Psi_{I\setminus\{j\}}\big|_{Z_j}\right)=r_I-d
\]
at every point of $Z_j$.  By the constant-rank theorem and a compact
exhaustion, the image of $\Psi_{I\setminus\{j\}}|_{Z_j}$ is a countable
union of compact $C^1$ coordinate patches of dimension
\[
 r_I-d\le q\,\# I-d=q(\# I-1)-2<q(\# I-1),
\]
so the remaining positive-time coordinates are polar by
Lemma~\ref{lem:onepathcapacity}.  If the anchored label does not belong to
$I$, the first paragraph applies to $\Psi_I$ directly.  A countable atlas
and compact exhaustion complete the proof.
\end{proof}

The absence argument is local and does not use scale invariance, which
enters only when local positive probability is upgraded to probability one.

\begin{theorem}[pattern trichotomy]\label{thm:trichotomy}
Let $\mathcal F$ be a scale-invariant family in $\R^d$, $d\ge3$, with a
rank-regular fixed-rank scale atlas of effective dimension $D$, and put
$\Delta_I=r_I-(d-2)\,\# I$ and $\Delta=D-k(d-2)$.
\begin{enumerate}[label=\normalfont(\alph*)]
\item If $\Delta>0$ and every proper nonempty subconfiguration has positive
excess, almost surely the trace of Brownian motion contains a member of
$\mathcal F$.
\item If $\Delta<0$, almost surely it contains no member of $\mathcal F$.
\item If $\Delta=0$, almost surely it contains no member of $\mathcal F$.
\end{enumerate}
\end{theorem}

\begin{proof}
If $\Delta>0$, choose a nonanchored parameter and a sufficiently small
normalised positive-excess compact chart around it; nonanchored parameters
exist because every coordinate map is a submersion, so each zero fibre
$\psi_i^{-1}(0)$ is a codimension-$d$ submanifold of Lebesgue measure zero,
and finitely many of them cannot cover a nonempty open chart.
Theorem~\ref{thm:localtrichotomy} gives a copy with positive probability.
The event that the trace contains some member of $\mathcal F$ is invariant
under spatial dilation, since the scaled path has trace $\lambda^{-1}B$ and
$\mathcal F$ is scale invariant; by Lemma~\ref{lem:scaleergodic} its
probability is zero or one, hence one.

If $\Delta\le0$, the full-label submap has nonpositive excess on every
chart.  Apply Proposition~\ref{prop:obstruction} to each chart, which
covers the strict and the critical inequality alike as well as
configurations through the starting point, and use countability of the
atlas.
\end{proof}

\begin{corollary}[sharp subconfiguration criterion]\label{cor:subcriterion}
For a scale-invariant family with a rank-regular fixed-rank atlas, a copy
occurs almost surely if and only if
\[
 \min_{\varnothing\ne I\subseteq\{1,\dots,k\}}
 \bigl(r_I-(d-2)\,\# I\bigr)>0 ;
\]
if the minimum is nonpositive, almost surely no copy occurs.
\end{corollary}

\begin{proof}
If every excess is positive, Theorem~\ref{thm:trichotomy}(a) applies.  If
some excess is nonpositive, the fixed-rank property gives, after the
chartwise permutation of labels, a nonpositive subconfiguration on every
chart; Proposition~\ref{prop:obstruction} excludes each chart, and
countability of the atlas excludes the whole family.
\end{proof}

\begin{corollary}[chartwise criterion]\label{cor:chartwise}
Let a scale-invariant family have a rank-regular scale atlas whose ranks and
parameter dimensions may vary between charts.  If one chart has positive
excess for every nonempty subconfiguration, then a copy occurs almost
surely.  If every chart has a nonpositive subconfiguration, then almost
surely no copy occurs.
\end{corollary}

\begin{proof}
An all-positive chart gives local positive probability by
Theorem~\ref{thm:localtrichotomy}, and scaling ergodicity upgrades the
occurrence event for the full family to probability one.  In the other
direction, Proposition~\ref{prop:obstruction} gives a null occurrence event
for each chart, and the atlas is countable.
\end{proof}

\section{Arithmetic progressions and other patterns}
\label{section:consequences}

For the families below, all of which have rank-regular fixed-rank scale
atlases, Corollary~\ref{cor:subcriterion} reduces occurrence to a finite
rank computation: the family occurs almost surely exactly when
$\Delta_I=r_I-(d-2)\,\# I>0$ for every nonempty $I$, and is almost surely
absent as soon as one $\Delta_I$ is nonpositive.  Every finite pattern $F$
is deterministic, so the exceptional null set may depend on $F$.  The
following lemma supplies the atlas in each case.

\begin{lemma}[atlases from global charts]\label{lem:atlasfromchart}
Let $\mathcal F$ be a nonempty scale-invariant family of $k$-point
configurations covered, after local labelling, by countably many $C^1$
embeddings $\Psi_\nu:U_\nu\to\operatorname{Conf}_k(\R^d)$ with $U_\nu$
open, such that every coordinate of $\Psi_\nu$ is a submersion and every
submap $\Psi_{\nu,I}$ has constant rank $r_I$ on $U_\nu$, with the same
$D_\nu=D$ and the same ranks $r_I$ for all $\nu$ after relabelling.  Then
$\mathcal F$ has a rank-regular fixed-rank scale atlas.
\end{lemma}

\begin{proof}
Each $U_\nu$ is second countable, hence exhausted by countably many
compact cubes, and the ranks are constant on a neighbourhood of each of
them; so the $\Psi_\nu$ form a countable rank-regular atlas, fixed-rank by
assumption.  It remains to refine a nonanchored parameter
$\omega_0\in U_\nu$ to a normalised rank chart.  The $k$ points of
$\Psi_\nu(\omega_0)$ are distinct and nonzero, so by continuity a
sufficiently small compact cube $K\ni\omega_0$ has labelled images lying
in pairwise disjoint balls away from the origin, each image of diameter at
most $a/10$ for a suitable $a>0$; this is \ref{N1}, and
Corollary~\ref{cor:Timpliesreg} gives \ref{N2}.
\end{proof}

\subsection{Arithmetic progressions in every dimension}\label{subsec:APall}

For $k\ge2$, the ordered $k$-term arithmetic progressions in $\R^d$ form
the scale-invariant family
\[
 \Phi_k(a,g)=\bigl(a+(i-1)g\bigr)_{i=1}^k,
 \qquad
 (a,g)\in\R^d\times(\R^d\setminus\{0\}),
\]
with open parameter domain of dimension $D=2d$; the length of the gap is a
genuine parameter, not a normalisation.

\begin{lemma}[ranks for progressions]\label{lem:APnondeg}
The progression family has a rank-regular fixed-rank scale atlas, with
\[
 r_I=
 \begin{cases}
 d,&\# I=1,\\
 2d,&\# I\ge2,
 \end{cases}
 \qquad
 \Delta_I=
 \begin{cases}
 2,&\# I=1,\\
 2d-(d-2)\,\# I,&\# I\ge2.
 \end{cases}
\]
\end{lemma}

\begin{proof}
The differential of the term with index $i$ is $A_i(u,v)=u+(i-1)v$, which
is surjective.  For $i_1\ne i_2$, the map
$(u,v)\mapsto\bigl(u+(i_1-1)v,u+(i_2-1)v\bigr)$ is invertible, since the
difference of its components gives $(i_2-i_1)v$ and then either component
gives $u$; so every submap with at least two terms has rank $2d$.  The full
ordered map is an embedding, because its first two coordinates recover
$(a,g)$ continuously, each coordinate is a submersion, because translation
is free, and dilation acts on parameters by $(a,g)\mapsto(\lambda a,\lambda g)$.
Lemma~\ref{lem:atlasfromchart} gives the atlas.  An unordered progression
has the two local labellings $(a,g)$ and $(a+(k-1)g,-g)$; this finite
redundancy changes neither dimensions nor ranks.
\end{proof}

For $\# I\ge2$ the excess decreases with $\# I$, and singleton excess is
$2$, so all subconfiguration excesses are positive exactly when the full
excess $\Delta=2d-k(d-2)$ is.  For $d\ge3$,
Corollary~\ref{cor:subcriterion} therefore says that a $k$-term progression
occurs almost surely exactly when $k<2d/(d-2)$, and is absent at and beyond
equality.

\begin{corollary}[longest progression in the spatial trace]
\label{cor:APall}
For $d\ge3$, the trace of Brownian motion in $\R^d$ almost surely contains
an arithmetic progression of length $\lceil2d/(d-2)\rceil-1$ and almost
surely contains none of length $\lceil2d/(d-2)\rceil$.  The maximal length
is thus
\[
 5\quad(d=3),\qquad
 3\quad(d=4,5),\qquad
 2\quad(d\ge6).
\]
\end{corollary}

\begin{proof}
The largest integer strictly smaller than $2d/(d-2)$ is
$\lceil2d/(d-2)\rceil-1$.  Corollary~\ref{cor:subcriterion} gives a
progression of this length, and Proposition~\ref{prop:obstruction} excludes
the next length, also when its excess is exactly zero; a longer progression
contains one of the excluded length.  Substituting $d=3,4,5$ and $d\ge6$
gives the displayed values.
\end{proof}

Together with Corollary~\ref{cor:planarAP}, this gives the progression
threshold in every dimension.  The critical equality $k(d-2)=2d$ occurs
only when $d-2$ divides $4$, that is, for a $6$-term progression in
$\R^3$, a $4$-term progression in $\R^4$ and a $3$-term progression in
$\R^6$; these are the progression cases that need the critical capacity
argument of Section~\ref{subsec:nonpositive}.

\subsection{Homothetic copies}\label{subsec:homothetic}

Fix $d\ge3$ and distinct points $F=\{f_1,\dots,f_k\}\subset\R^d$, $k\ge2$.
Positive homothetic copies of $F$ are parametrised by
\[
 \Phi_F(a,t)=\bigl(a+tf_i\bigr)_{i=1}^k,
 \qquad (a,t)\in\R^d\times(0,\infty),
\]
so $D=d+1$.  The differential of the $i$-th coordinate,
$A_i(u,s)=u+sf_i$, has rank $d$, and for $\# I\ge2$ the combined map has
rank $d+1$: if $u+sf_i=u+sf_l=0$ with $i\ne l$, then $s(f_i-f_l)=0$, so
$s=0$ and $u=0$.  Hence singleton excesses equal $2$ and
$\Delta_I=d+1-(d-2)\,\# I$ for $\# I\ge2$, so all subconfiguration excesses
are positive exactly when the full excess is.

\begin{corollary}[homothetic copies]\label{cor:homothetic}
For $d\ge3$ and every fixed set $F\subset\R^d$ of $k\ge2$ distinct points,
the trace of Brownian motion almost surely contains a positive homothetic
copy of $F$ if
\[
 k<\frac{d+1}{d-2},
\]
and almost surely contains none if $k\ge(d+1)/(d-2)$.
\end{corollary}

\begin{proof}
The full ordered map is an embedding: two labelled coordinates
$x_i=a+tf_i$ and $x_j=a+tf_j$ give
\[
 t=\frac{\langle x_i-x_j,f_i-f_j\rangle}{|f_i-f_j|^2},
 \qquad a=x_i-tf_i,
\]
where $\langle\cdot,\cdot\rangle$ is the Euclidean inner product.  Each
coordinate is a submersion because translation is free, the submap ranks
computed above are constant, and dilation acts on parameters by
$(a,t)\mapsto(\lambda a,\lambda t)$.  Lemma~\ref{lem:atlasfromchart} gives
a rank-regular fixed-rank scale atlas, and Corollary~\ref{cor:subcriterion}
gives the conclusion.
\end{proof}

Explicitly, for each prescribed $F$: in $\R^3$ a positive homothetic copy
occurs almost surely exactly when $2\le k\le3$, four points being critical
since $D=4$; in $\R^4$ exactly when $k=2$; and for $d\ge5$ never.
Progressions have the larger parameter dimension $2d$ because their
direction is also free.  Rotations are not part of a homothety, so the
absence of prescribed two-point homothetic copies for $d\ge5$ does not
conflict with the following remark, in which the direction of the pair is
free.

\begin{remark}[distance set]\label{rem:distance}
Almost surely the distance set $\mathcal D(B):=\{|x-y|:x,y\in B\}$ equals
$[0,\infty)$.  Indeed, $W(0)=0$, $t\mapsto|W(t)|$ is continuous, and
$|W(t)|\to\infty$ by transience \cite[Thm.~3.20]{PM}, so the intermediate
value theorem gives $\{|W(t)-W(0)|:t\ge0\}=[0,\infty)$.  This says nothing
about prescribed gap vectors.
\end{remark}

\subsection{Similar copies}\label{subsec:similar}

Let $F=\{f_1,\dots,f_k\}\subset\R^d$ consist of distinct points and let
$m=\dim\operatorname{aff}F\ge1$.  Orientation-preserving similar copies
have the form
\[
 \Phi_F(a,t,Q)=\bigl(a+tQf_i\bigr)_{i=1}^k,
 \qquad
 (a,t,Q)\in\R^d\times(0,\infty)\times SO(d),
\]
where $O(d)=\{Q:Q^TQ=\mathrm{Id}\}$ is the orthogonal group and
$SO(d)=\{Q\in O(d):\det Q=1\}$.  Translate $F$ so that $f_1=0$ and put
$V=\operatorname{span}F$.  Two parameter triples give the same ordered
configuration if and only if their translations and scales agree and their
rotations agree on $V$: the coordinate labelled by $f_1=0$ gives $a=a'$;
for $f_i\ne0$, $|tQf_i|=|t'Q'f_i|$ and positivity of the scales give
$t=t'$; and the remaining equalities say $Qv=Q'v$ for every $v\in V$.  Thus
the rotations fixing all labelled points form the pointwise stabiliser
\[
 \{(0,1,Q):Q|_V=\mathrm{Id}_V\}\cong SO(d-m),
\]
the rotation group of $V^\perp$ (trivial when $d-m\le1$), of dimension
$\binom{d-m}{2}$.  Taking local coordinates on $SO(d)$ modulo this closed
subgroup removes the redundant rotations, and the similarity family has
dimension
\begin{equation}\label{eq:similarD}
 D_m=d+1+\binom d2-\binom{d-m}{2}
 =d+1+\tfrac12m(2d-m-1).
\end{equation}
Allowing $Q\in O(d)$ adds at most one further orbit, with the same ranks
and the same conclusion; passing to unordered configurations adds only
finite labelling redundancy.

Put $p=d+1+\binom d2$.  At $t=1$ and $Q=\mathrm{Id}$, the unreduced
differential of the $i$-th coordinate is
\[
 \widetilde A_i(u,s,X)=u+sf_i+Xf_i,
 \qquad
 u\in\R^d,\quad s\in\R,\quad X\in\mathfrak{so}(d),
\]
where $\mathfrak{so}(d)=\{X\in\R^{d\times d}:X^T=-X\}$ is the tangent
space of $SO(d)$ at the identity.

\begin{lemma}[rank of a similarity subfamily]\label{lem:simrank}
For a singleton, $r_{\{i\}}=d$.  If $\# I\ge2$ and
$e_I=\dim\operatorname{aff}\{f_i:i\in I\}$, then the corresponding submap
of the nonredundant similarity family has rank
\[
 r_I=p-\binom{d-e_I}{2}
 =d+1+\tfrac12e_I(2d-e_I-1).
\]
In particular, $r_{\{1,\dots,k\}}=D_m$.
\end{lemma}

\begin{proof}
The kernel of $\widetilde A_I$ consists of the triples with
$(s\,\mathrm{Id}+X)(f_i-f_l)=0$ for all $i,l\in I$.  For a nonzero vector
$v$ in the $e_I$-dimensional direction space of the affine span,
$0=\langle(s\,\mathrm{Id}+X)v,v\rangle=s|v|^2$ because $X$ is
antisymmetric, so $s=0$.  Then $X$ vanishes on the direction space, whose
orthogonal complement is invariant by skew-symmetry, and the free block
there is an arbitrary element of $\mathfrak{so}(d-e_I)$; finally
$u=-Xf_i$ for any $i\in I$, independently of $i$ because $X$ kills all
differences $f_i-f_l$.  Hence $\dim\ker\widetilde A_I=\binom{d-e_I}{2}$.
Let $K_F\subseteq K_I$ be the kernels for the full configuration and for
$I$, and let $T$ be a linear complement of $K_F$ in the unreduced tangent
space; $T$ supplies the nonredundant parameter directions.  Since
$K_I=K_F\oplus(K_I\cap T)$,
\begin{align*}
 \operatorname{rank}(D\Phi_I|_T)
 &=\dim T-\dim(K_I\cap T)\\
 &=(p-\dim K_F)-(\dim K_I-\dim K_F)
 =p-\binom{d-e_I}{2}.
\end{align*}
A singleton has rank $d$ because translation is free.
\end{proof}

These local coordinates, with the redundant rotations removed, give
countably many embeddings covering the nonredundant similarity family.  At
a general parameter $(a,t,Q)$, applying $Q^{-1}$ to all target coordinates
and rescaling the tangent variables by $t$ reduces every rank computation
to the one at $(0,1,\mathrm{Id})$, so all submap ranks are constant, and
the free translation makes every coordinate a submersion; dilation sends
$(a,t,Q)$ to $(\lambda a,\lambda t,Q)$.  Lemma~\ref{lem:atlasfromchart}
therefore gives a rank-regular fixed-rank scale atlas.

Every non-singleton subconfiguration has excess at least the full excess.
Let $I$ contain $n\ge2$ points and have affine dimension $e=e_I$.  Adding
the remaining $k-n$ points raises the affine dimension by at most $k-n$,
so $n\le k-(m-e)$.  With $q=d-2$, $\Delta_F=D_m-qk$ and
$r_e=d+1+\frac12e(2d-e-1)$, Lemma~\ref{lem:simrank} gives
$r_e-D_m=-\frac12(m-e)(2d-m-e-1)$, hence
\begin{align*}
 \Delta_I=r_I-qn
 &\ge r_e-q\bigl(k-(m-e)\bigr)
 =\Delta_F+\tfrac12(m-e)(m+e-3)
 \ge\Delta_F ,
\end{align*}
the last term being zero if $e=m$ and nonnegative if $e<m$, since then
$m\ge2$ and $e\ge1$.  Singleton excesses equal $d-(d-2)=2$.

\begin{theorem}[threshold for similar copies]\label{thm:similar}
Fix $d\ge3$ and a finite set $F\subset\R^d$ of $k\ge2$ distinct points with
affine dimension $m\ge1$, and put $D_m=d+1+\frac12m(2d-m-1)$.  The trace
of Brownian motion almost surely contains a similar copy of $F$ if
\[
 k<\frac{D_m}{d-2},
\]
and almost surely contains none if $k\ge D_m/(d-2)$.
\end{theorem}

\begin{proof}
If the full excess $D_m-(d-2)k$ is positive, every subconfiguration has
positive excess by the computation above, and
Corollary~\ref{cor:subcriterion} gives occurrence.  If it is nonpositive,
Proposition~\ref{prop:obstruction} gives absence, including at equality.
\end{proof}

Put
\[
 k_{\max}(d,m)=\left\lceil\frac{D_m}{d-2}\right\rceil-1 .
\]
Then each prescribed configuration of affine dimension $m$ with
$m+1\le k\le k_{\max}(d,m)$ almost surely has a similar copy on the trace,
and each prescribed one with $k>k_{\max}(d,m)$ almost surely has none; if
$k_{\max}(d,m)<m+1$, no prescribed configuration of affine dimension $m$
occurs.  The quantifier is for each fixed $F$, with an exceptional null
set depending on $F$; nothing is asserted about all configurations with
more than $k_{\max}(d,m)$ points simultaneously.  The formula contains the
progression threshold: for $m=1$, $D_1=2d$, and the similarity family of a
$k$-term progression is the progression family up to the finite reversal
of its labelling, so a $k$-term progression occurs exactly when
$k<2d/(d-2)$.  In $\R^3$, $D_1=6$ and $D_2=D_3=7$: each prescribed
collinear set of at most five points and each prescribed noncollinear set
of at most six points has a similar copy on the trace, while six collinear
points and seven noncollinear points are the two critical cases.

The maximal-rank condition \eqref{eq:transversality} is stronger than
occurrence.  In $\R^4$, take five affinely independent points: $m=4$,
$D_m=11$, and any three points span a plane, so their submap has rank
\[
 r_2=5+\tfrac12\cdot2(8-2-1)=10<11=\min(3d,D_m)
\]
and \eqref{eq:transversality} fails.  All excesses are nevertheless
positive, the full excess being $11-2\cdot5=1$, so every fixed affinely
independent five-point set in $\R^4$ has a similar copy on the trace almost
surely.

\begin{remark}[corners]\label{rem:corners}
In $\R^3$, the right isosceles corner $\{0,e_1,e_2\}$ has affine dimension
$2$, full rank $7$ and excess $7-3=4$, and every pair has rank $6$ and
excess $6-2=4$; Theorem~\ref{thm:similar} therefore gives a translated,
rotated and uniformly scaled copy on the trace.  The same computation
applies to every fixed noncollinear triangle: an aperture and a side-length
ratio may be prescribed in advance, and the theorem supplies some
orientation of that shape.  The quantifier is for each fixed shape; it does
not assert simultaneous realisation of an uncountable family of shapes.
Prescribing absolute side lengths removes dilation and is not covered by
this scale-invariant theorem.
\end{remark}

\section*{Acknowledgments}
A.~Y. was supported in part by the Natural Sciences and Engineering
Research Council of Canada, NSERC (GR030571 and GR030540).

The authors used ChatGPT to assist with the organisation and exposition of
the manuscript.  All mathematical statements, proofs, references and final
text were reviewed and verified by the authors, who take full
responsibility for the content.

\end{document}